\documentclass[11pt,reqno]{amsart}
\usepackage{fullpage}
\usepackage{amsmath,amsthm,verbatim,amssymb,amsfonts,amscd, graphicx,bbm}
\usepackage{graphics}
\usepackage{mathtools}
\usepackage{mathrsfs}
\usepackage{eufrak}
\usepackage{esint}

\usepackage{hyperref}

\newtheorem{theorem}{Theorem}

\newtheorem{lemma}[theorem]{Lemma}
\newtheorem{corollary}[theorem]{Corollary}

\theoremstyle{definition}
\newtheorem{remark}[theorem]{Remark}

\newcommand{\cref}[1]{Corollary~\ref{c.#1}}

\numberwithin{equation}{section}
\numberwithin{theorem}{section}

\newcommand{\Z}{\mathbb{Z}}

\newcommand{\R}{\mathbb{R}}

\newcommand{\cS}{\mathcal{S}}
\newcommand{\cD}{\mathcal{D}}
\newcommand{\scD}{\mathscr{D}}
\newcommand{\cK}{\mathcal{K}}
\newcommand{\cR}{\mathcal{R}}

\renewcommand{\tilde}{\widetilde}

\newcommand{\ol}{\overline}

\newcommand{\eps}{\varepsilon}

\newcommand{\fp}{\mathfrak{p}}

\title{A unified homogenization approach for the Dirichlet problem in Perforated Domains}

\author{Wenjia Jing}
\address{Yau Mathematical Sciences Center,
  Tsinghua University,
  No.1 Tsinghua Yuan,
  Beijing 100084, People's Republic of China}
\email{wjjing@tsinghua.edu.cn}

\date{January 24, 2019}

\begin{document}

\begin{abstract}

We revisit the periodic homogenization of Dirichlet problems for the Laplace operator in perforated domains, and establish a unified proof that works for different regimes of hole-cell ratios, that is the ratio between the scaling factor of the holes and that of the periodic cells. The approach is then made quantitative and it yields correctors and error estimates for vanishing hole-cell ratios. For positive volume fraction of holes, the approach is just the standard oscillating test function method; for vanishing volume fraction of holes, we study asymptotic behaviors of a properly rescaled cell problems and use them to build oscillating test functions. Our method reveals how the different regimes are intrinsically connected through the cell problems and the connection with periodic layer potentials. 

\smallskip


\smallskip

\noindent{\bf Mathematics subject classification (MSC 2010)}: 35B27, 35J08

\end{abstract}

\maketitle


\section{Introduction}

In this article we revisit the periodic homogenization problem in perforated domains that are formed by removing a periodic array of small holes from a fixed open bounded and connected domain with regular boundary. For simplicity, we restrict the analysis to the Dirichlet problem for the Laplace operator on such domains. Our goal is to develop an adaptive approach that yields both qualitative homogenization and also error estimates, and it is adaptive with respect to the ratio between the size of the holes and the periodicity of the array. The problem of interest is, hence, to study the large scale behavior of $u^{\eps,a_\eps}$ that solves the following problem:
\begin{equation}
\label{eq:ppeps}
\left\{
\begin{aligned}
&-\Delta u^{\eps,a_\eps}(x) = g(x), &\quad &x \in D^{\eps,a_\eps},\\
&u^{\eps,a_\eps}(x) = 0, &\quad &x \in \partial D^{\eps,a_\eps}.
\end{aligned}
\right.
\end{equation}  
Here $D^{\eps,a_\eps} \subset \R^d$ denotes the perforated domain as described above, $\eps$ and $a_\eps$ are  small positive numbers which represent roughly the sizes of the periodicity and the holes respectively. The source term $g$ is assumed to be sufficiently regular, say in $L^2(D)$.

We now set up the mathematical model for $D^{\eps,a_\eps}$ here and use it throughout the paper. Let $D \subset \R^d$, $d \ge 2$, be an open bounded and connected domain, with $C^{1,\alpha}$ boundary $\partial D$ for some $\alpha \in (0,1)$. To obtain $D^{\eps,a_\eps}$, we remove from $D$ a periodic array of small sets. Those sets are rescaled from unit ones, so we start with the scenario at the unit scale. Let $Q := (-\frac12,\frac12)^d$ be the unit cube centered at the origin with side length one. Let $T \subset Q$ be an open set satisfying
\begin{equation}
\label{eq:refhole}
\overline{B}_{\frac1{16}}(0) \subset T \subset \overline{T} \subset B_\frac14(0),
\end{equation}
and with $C^{1,\alpha}$ boundary $\partial T$. For each $\eta \in (0,1]$, we set $Y_{f,\eta} := Q\setminus \eta \ol T$. Then $Y_{f,\eta}$ is the part of $Q$ with the set $\eta \ol T$ removed; for $\eta = 1$, we use $Y_f$ for $Y_{f,1}$. The subscript ``$f$" refers to ``fluid" or material part of the domain.

We then take copies of $Y_{f,\eta}$, and glue them together to form the perforated whole space
\begin{equation}
\R^d_{f,\eta} := \cup_{k \in \Z^d} \left( k+ Y_{f,\eta}\cup \partial Q \right) = \R^d \setminus \cup_{k \in \Z^d} \left( k + \eta \ol T\right).
\end{equation}
The second representation above says, $\R^d_{f,\eta}$ is constructed by removing from $\R^d$ the $\Z^d$-translated copies of the model hole $\eta \ol T$. We then rescale $\R^d_{f,\eta}$ by a small positive number $\eps$ and get $\eps \R^d_{f,\eta}$. In this scaled perforated whole set, the holes are rescaled from $T$ by a factor of $\eta \eps$, and each hole is enclosed by a periodic cell rescaled from $Q$ by a factor of $\eps$. 

The perforated domain $D^{\eps,a_\eps}$ of this paper is then defined by
\begin{equation*}
D^{\eps,a_\eps} := D \cap (\eps \R^d_{f,\eta_\eps}), \qquad\text{with} \quad \eta_\eps := a_\eps/\eps.
\end{equation*}
In other words, $D^{\eps,a_\eps}$ is obtained by cutting the piece of $\eps \R^d_{f,\eta_\eps}$ that is inside the set $D$ fixed earlier. The boundary of $D^{\eps,a_\eps}$ hence consists of $\partial D$ and the union of the boundaries of the holes. To avoid non-Lipschitz domains we modify the above definition a little bit: for those $k \in \Z^d$ such that $\partial D \cap \eps (k+\partial Q)$ is nonempty, we add the set $\eps(k+\eta\ol T)$ back to $D^{\eps,a_\eps}$, so the holes inside $D^{\eps,a_\eps}$ are separated from the outer boundary $\partial D$ by $\eps$. 

Through out the paper, we assume:
\begin{enumerate}
	\item[(P1)] For each $\eps \in (0,1)$, $a_\eps$ is chosen in $(0,\eps)$, and $D^\eps = D^{\eps,a_\eps}$ is constructed as above. In particular, the reference hole $T$ satisfies \eqref{eq:refhole} and $\eta_\eps = \frac{a_\eps}{\eps}$ is the ratio between the scaling factors of the holes and the periodic cells.
	\item[(P2)] The source term $g$ belongs to $L^2(D)$.
\end{enumerate} 
In the rest of the paper, we simplify the notations $D^{\eps,a_\eps}$ and $u^{\eps,a_\eps}$ to $D^\eps$ and $u^\eps$.

The goal of homogenization is to find the asymptotic behavior of $u^{\eps}$ as $\eps$ goes to $0$. Let us first consider the setting where $a_\eps = \eps$. In this case, the domain $D^\eps$ has two distinguished scales; the outer boundary $\partial D$ has length scale one, and the boundary of holes has scale $\eps$. We use the formal two-scale expansion and consider the ansatz
\begin{equation*}
u^\eps = u_0(x,\frac x\eps) + \eps u_1(x,\frac x\eps) + \eps^2 u_2(x,\frac x\eps) + \cdots
\end{equation*} 
where $u_i: \R^d\times Y_f \to \R$, $i = 0,1,2,\cdots$, satisfies periodic condition at the outer boundary $\partial Y_f \cap \partial Q$ and the Dirichlet boundary condition $u_i = 0$ at $\partial Y_f \cap Q$. Apply this ansatz in \eqref{eq:ppeps}, replace $\nabla$ by $\nabla_x + \frac{1}{\eps}\nabla_y$, and identify terms of the same order in $\eps$. We get,
\begin{equation*}
\begin{aligned}
&-\Delta_y u_0(x,y) = 0,\\
&-\Delta_y u_1(x,y) + 2\nabla_x\cdot\nabla_y u_0(x,y) = 0,\\
&-\Delta_y u_2(x,y) + 2\nabla_x\cdot\nabla_y u_1(x,y) = g(x).
\end{aligned}
\end{equation*}
We view those equations as posed on $Y_f$ with periodic boundary condition at $\partial Q$ and Dirichlet condition at $\partial Y_f \cap Q$, and the $x$-variable as a parameter. Then $u_0$ has to be zero, and so does $u_1$. The equation for $u_2$ reduces to $-\Delta_y u_2(x,y) = g(x)$. By linearity and uniqueness, we get $u_2(x,y) = g(x)\chi(y)$, where $\chi : \R^d_f \to \R$ is $\Z^d$-periodic and solves the \emph{cell problem}: 
\begin{equation}
\label{eq:cellprob}
\left\{
\begin{aligned}
-\Delta_y \chi(y) = 1, \qquad &y \in Y_f\\
\chi(y) = 0, \qquad &y \in \partial T.
\end{aligned}
\right.
\end{equation}
Elliptic theory shows that $\chi$ is uniquely determined. Then formally we have $u^\eps(x) \approx \eps^2 g(x)\chi(\frac{x}{\eps})$, and it follows that $u^\eps/\eps^2$ converges weakly in $L^2$ to $c_* g$ where $c_* = \int_{Y_f} \chi$. We conclude that $u^\eps$ is of order $\eps^2$ and the limit of $u^\eps/\eps^2$ is given by an algebraic equation. In some sense, this is analogous to the homogenization derivation of Darcy's law for Stokes' problem in periodically perforated domain. The formal two-scale expansion derivations for Stokes' system goes back at least to Keller \cite{Keller} and Sanchez-Palencia \cite{Sanchez}, and the rigorous proof was first obtained by Tartar \cite{Tartar}. Tartar's proof is based on oscillating test functions built from rescalings of the cell problem. The approach of Tartar was then generalized to the evolutional setting by Mikeli\u{c} \cite{Mikelic} and to the compressible setting by Masmoudi \cite{Masmoudi}. We refer to the books \cite{BLP} and \cite{Jikov} for more comprehensive introductions to the homogenization theory.

As long as different scaling for holes and cells is considered, Cioranescu and Murat were the first to identify the critical scaling ratio between $a_\eps$ and $\eps$; see \cite{CioMur-1,CioMur-2}. They found that the critical relative order (with respect to $\eps$) for $a_\eps$ is given by 
\begin{equation}
\label{eq:acritical}
a^\eps_* \sim \eps^{\frac{d}{d-2}}, \quad d\ge 3, \qquad \text{and} \qquad 
\log a^\eps_*  \sim -\frac{1}{\eps^2}, \quad d = 2.
\end{equation}
Under those critical scalings, say when equality holds above, the zero extension of $u^{\eps}$ is of order one and converges weakly in $H^1(D)$ to $u$, the solution of the effective problem
\begin{equation}
\label{eq:pp_c}
\left\{
\begin{aligned}
&-\Delta u(x) + \mu_* u(x) = g(x), &\quad &x \in D,\\
&u(x) = 0, &\quad &x \in \partial D,
\end{aligned}
\right.
\end{equation}
where $\mu_*$ is a positive constant. In fact, $\mu_*$ is the Newtonian capacity of the set $T$ for $d\ge 3$ and it equals to $1/2\pi$ (which is the logarithmic capacity) for $d=2$; see Section \ref{sec:mainresults}. On the other hand, for $a_\eps \ll a^\eps_*$, the limit of $u^\eps$ is the solution to
\begin{equation}
\label{eq:pp_subc}
\left\{
\begin{aligned}
&-\Delta u(x) = g(x), &\quad &x \in D,\\
&u(x) = 0, &\quad &x \in \partial D.
\end{aligned}
\right.
\end{equation}
In other words, in the setting of $a_\eps \ll a^\eps_*$, the homogenized problem does not see the holes. Due to this comparison, Cioranescu and Murat named the term $\mu_* u$ in the critical scaling the ``\emph{strange term coming from nowhere}".

\medskip

\noindent{\bf Scaling regimes.} The critical scaling $a^\eps_*$ comes from the application of the Poincar\'e inequality \eqref{eq:poincare} that is recorded in the Appendix. Given the scaling factors $\eps$ for the periodic cells and $a_\eps$ for the hole, let $\eta = \eta(\eps) = a_\eps/\eps$. We then follow Allaire \cite{Allaire91-1} and define the parameter
\begin{equation}
\label{eq:sige}
\sigma_\eps = \eps \eta^{-\frac{d-2}{2}} \quad d\ge 3, \qquad\text{and}\qquad 
\sigma_\eps = \eps|\log \eta|^\frac12 \quad d=2.
\end{equation}
The critical scaling $a^\eps_*$ is chosen to make $\sigma_\eps$ be of order one. In the rest of the paper, by \emph{positive (limiting) volume fraction} regime we mean $a_\eps \sim \eps$, and \emph{vanishing volume fraction} refers to the situation of $a_\eps/\eps \to 0$. For the latter case, we consider three sub-cases. If $a_\eps \sim a^\eps_*$, or equivalently $\sigma_\eps \sim 1$, we call it the \emph{``critical hole-cell ratio''} setting; similarly, \emph{``sub-critical hole-cell ratio''} refers to the situation of $a_\eps \ll a^\eps_*$ (note that $\sigma_\eps \gg 1$), and \emph{``super-critical hole-cell ratio''} refers to the case when $a^\eps_* \ll a_\eps$, or $\sigma_\eps \ll 1$. Those terminologies are used throughout the paper.

In the pioneering work of Cioranescu and Murat \cite{CioMur-1,CioMur-2}, a general framework was developed for the setting of vanishing volume fractions of holes for the Dirichlet problem of the Laplace operator. Moreover, a corresponding framework for quantification was established by Kacimi and Murat \cite{Kacimi_Murat}. The framework was later extended by Allaire to Stokes and Navier-Stokes problems \cite{Allaire91-1,Allaire91-2}, to the obstacle problems by Caffarelli and Mellet \cite{CafMel}, and to random settings by Ho\`ang \cite{Hoang-2,Hoang}; see also \cite{FeiLu-1} and see \cite{GiuHof-1,GiuHof-2} for randomly perforated domains based on Poisson piont processes.

The framework in \cite{CioMur-1} was set for holes with vanishing volume fractions and there is no reference to the cell problems. The latter, on the other hand, is central for the framework of Tartar for the more classical setting of $a_\eps \sim \eps$. A natural question arises to find the connection between the two frameworks. As far as we know, Allaire was the first to consider such a relation, and in \cite{Allaire-3} he found that the strange term in the critical setting (which corresponds to the so-called Brinkman's law) is the inverse of the permeability tensor in the super-critical setting with vanishing volume fraction of holes (which corresponds to the Darcy's laws), and this permeability tensor is the limit of a family of such tensors associated to the cell problems posed on a large periodic cell with a model hole removed. 

Inspired by this work of Allaire, we develop in this paper a unified approach for the homogenization of \eqref{eq:ppeps} that is based on the oscillating test function method and on the analysis of cell problems, and it works for all regimes of hole-cell ratios. For vanishing volume fraction of holes, we need to consider rescaled cell problems posed in a large periodic box whose side length is the inverse of the hole-cell ratio, and with the model hole removed. The large box limit of the cell problems is obtained by the compactness method very similar to Allaire \cite{Allaire-3}. A major profit of this new adaptive approach is, it can be easily quantified. The key tool we use for the quantifications is the large box limit of periodic potential theory; our results on this is interesting in its own rights.

Before concluding this introduction, we remark that the Dirichlet condition is crucial to make the scaling differences in the homogenization. If, for instance, the Neumann boundary $\nu\cdot \nabla u^\eps = 0$ is equipped in \eqref{eq:ppeps}, the homogenization procedure will not distinguish a critical scaling for $a_\eps$. This is because $u^\eps$ will not have a natural extension that is small somewhere, and the Poincar\'e inequality \eqref{eq:poincare} will be irrelevant; see e.g. \cite{AllMur} for a detailed treatment of the Neumann problem. We also refer to \cite{Monsur,AGGJS,J-CMS} for homogenization with transmission conditions across the boundaries of the holes.

The rest of the paper is organized as follows: We state the main theorems in the next section, which include not only qualitative homogenization but also error estimates for vanishing volume fractions. In section \ref{sec:otest} we prove the homogenization results by an adaptive oscillating test function method that works for various regimes of hole-cell ratios. The key large box limit of the cell problem studied in section \ref{sec:cell}. In section \ref{sec:error} we quantify the homogenization and obtain error estimates. Although the unified approach for homogenization works for all spatial dimensions $d \ge 2$, the two dimensional setting is different in terms of technical details. Hence, the proofs in sections \ref{sec:otest} and \ref{sec:error} are restricted to $d\ge 3$ and in section \ref{sec:twod} we present the main modifications that are needed for $d=2$. In the Appendix, we record facts on a Poincar\'e inequality and on potential theory of the Laplace operator; they are the key tools in our analysis.

\noindent{\bfseries Notations.} We list some special notations that will be used in the paper. Given $a,b\in [0,\infty)$, we use $a\wedge b$ for the minimum of the two numbers, and $a\vee b$ for the maximum. We use $\scD^{1,2}(\R^d)$ for the completion of smooth and compactly supported functions with respect to the norm $\|\nabla \varphi\|_{L^2(\R^d)}$. Let $E$ be an open set of $\R^d$ with smooth boundary, if $f$ is in the Sobolev space $H^1(D)$ and the trace of $f$ on $\partial E$ vanishes, we denote by $\tilde f$ the zero extension of $f$; that is $\tilde f: \R^d \to \R$ with $\tilde f = f$ on $E$ and $\tilde f = 0$ in $\R^d\setminus \ol E$. For a measurable function $f \in L^1(E;\R)$ where $E$ is a bounded open set of $\R^d$, we denote by $\langle f\rangle_E$ or $\fint_E f$ the average $\frac{1}{|E|}\int_E f$; similarly, $\langle f\rangle_S$ and $\fint_S f$ both denotes the average of $f$ over a hypersurface $S$ in $\R^d$, where the induced surface measure is used in the integral. We denote cubes with unit side length by $Q$, and by $rQ$ or $Q_r$ the dilation of $Q$ with respect to its center by a factor of $r$; similarly, $B$, $B_r$ and $rB$ are the corresponding notations for balls.

\section{Preliminaries and the Main results}
\label{sec:mainresults}

\medskip

We start this section by the energy estimates for \eqref{eq:ppeps}. Throughout the paper only consider the $L^2$ setting. Under the usual assumptions, for each fixed $\eps$ and $a_\eps$, the basic elliptic PDE theory yields that \eqref{eq:ppeps} admits a unique solution $u^\eps \in H^1_0(D^\eps)$. Let $\tilde u^\eps$ be the extension of $u^\eps$ to $D$ by zero. Then the following estimates hold.

\begin{lemma}
\label{lem:u_ee} Assume {\upshape(P1)} and {\upshape(P2)}. Then there exists a positive constant $C$ depending only on the model hole $T$, the whole set $D$ and the spatial dimension $d$, such that for all $\eps$ and $a_\eps$, the unique solution $u^\eps$ of \eqref{eq:ppeps} satisfies
\begin{equation}
\label{eq:ubdd}
\|\tilde u^\eps\|_{L^2(D)} \le C(1\wedge \sigma_\eps^2)\|g\|_{L^2(D)} , \quad \|\nabla \tilde u^\eps\|_{L^2(D)} \le C(1\wedge \sigma_\eps)\|g\|_{L^2(D)}.
\end{equation}
\end{lemma}

\begin{proof} Integrate the equation \eqref{eq:ppeps} against $u^\eps$ itself, we get
\begin{equation*}
\int_D |\nabla \tilde u^\eps|^2 = \int_{D^\eps} |\nabla u^\eps|^2 = \int_{D^\eps} g u^\eps = \int_D g\tilde u^\eps.
\end{equation*}
On the one hand, since $\tilde u^\eps \in H^1_0(D)$ and $D$ has bounded diameter, by the Poincar\'e inequality we can find some constant $C$ depending only on $D$ and $d$ such that $\|\tilde u^\eps\|_{L^2(D)} \le C\|\nabla \tilde u^\eps\|_{L^2(D)}$. We then deduce that
\begin{equation*}
\|\nabla \tilde u^\eps\|_{L^2(D)} \le C\|g\|_{L^2(D)} \qquad\text{and}\qquad
\|\tilde u^\eps\|_{L^2(D)} \le C^2\|g\|_{L^2(D)}.
\end{equation*}
On the other hand, the integral of $|\tilde u^\eps|^2$ and $|\nabla \tilde u^\eps|^2$ over the whole space $D$ is a sum of their integrals over cubes in $\{Q_\eps^k = \eps(k+Q) \,:\, Q_\eps^k \cap D \ne \emptyset, k \in \Z^d\}$. Inside each $Q_\eps^k$, in view of the fact that $\ol B_{1/16} \subset T$, we check that $\tilde u^\eps = 0$ in a ball whose radius is of order $a_\eps$. Hence, $\tilde u^\eps$ satisfies the conditions of the Poincar\'e inequality in \eqref{eq:poincare}. Applying the inequality, we get
\begin{equation*}
\|\tilde u^\eps\|^2_{L^2(D)} = \sum_{\{k\,:\, Q_\eps^k \cap D \ne \emptyset\}} \|\tilde u^\eps\|^2_{L^2(Q_\eps^k)} \le \sum_{k} C\sigma_\eps^2 \|\nabla \tilde u^\eps\|^2_{L^2(Q_\eps^k)} = C\sigma_\eps^2 \|\nabla \tilde u^\eps\|^2_{L^2(D)}.
\end{equation*}
Apply this inequality to the integral at the beginning of the proof, we obtain
\begin{equation*}
\|\nabla \tilde u^\eps\|_{L^2(D)} \le C\sigma_\eps \|g\|_{L^2(D)} \qquad\text{and}\qquad
\|\tilde u^\eps\|_{L^2(D)} \le C^2\sigma_\eps^2\|g\|_{L^2(D)}.
\end{equation*}
The constant $C$ above only depends on $T$ and $d$. The desired estimate \eqref{eq:ubdd} then follows.
\end{proof}

\subsection{Homogenization results}

Lemma \ref{lem:u_ee} provides bounds on $\tilde u^\eps$, and the bounds depend on the hole-cell ratios, and $a^\eps_*$ is the critical scaling for the size of holes below which the Poincar\'e inequality loses its force and $\|\tilde u^\eps\|_{L^2(D)}$ is no longer of small order. From the proof of the lemma, we see that those regime dependent phenomena come from the Dirichlet boundary condition at the boundary of holes. 

We expect that the homogenization of $\tilde u^\eps$ also depends on the hole-cell ratios. We define 
\begin{equation}
\label{eq:chat}
c_* = \left\{ \begin{aligned}
&\langle \chi \rangle_{Q}, &\qquad &\text{if } a_\eps = \eps,\\
&\frac{1}{\mathrm{Cap}(T)}, &\qquad &\text{if } \lim_{\eps \to 0} a_\eps/\eps = 0 \text{ and } d\ge 3\\
&\frac{1}{2\pi}, &\qquad &\text{if } \lim_{\eps \to 0} a_\eps/\eps = 0 \text{ and } d \ge 2.\\
\end{aligned}
\right.
\end{equation}
Here, $\mathrm{Cap}(T)$ denotes the Newtonian capacity of the set $T \subset \R^d$, $d\ge 3$, which is an importance concept in potential theory. In two dimensions, $2\pi$ is related to the Logarithmic capacity of $T$. It will be clear in sections \ref{sec:cell} and \ref{sec:error} how the capacity of $T$ enters our analysis. 

The first main theorem of the papers is the qualitative homogenization of \eqref{eq:ppeps}; the scaling factors for $\tilde u^\eps$ come from the energy estimates.

\begin{theorem}
\label{thm:homog}
Assume {\upshape(P1)} and {\upshape(P2)}, let $\sigma_\eps$ be defined as in \eqref{eq:sige} and let $c_*$ be defined as in \eqref{eq:chat}. Then we have the following results.
\begin{enumerate}
\item[(1)] If \, $\lim_{\eps \to 0} \sigma_\eps = 0$, i.e. $a^\eps_* \ll a_\eps \le \eps$, then $\tilde u^\eps/\sigma_\eps^2$ converges weakly in $L^2(D)$ to $u := c_* g$.
\item[(2)] If\, $\lim_{\eps \to 0} \sigma_\eps = 1$, i.e. $a_\eps \sim a_*^\eps$, then $\tilde u^\eps$ converges weakly in $H^1(D)$ and strongly in $L^2(D)$ to the unique solution of \eqref{eq:pp_c}, where $\mu_* = c_*^{-1}$.
\item[(3)] If\, $\lim_{\eps \to 0} \sigma_\eps = \infty$, i.e. $a_\eps \ll a_*^\eps$, then $\tilde u^\eps$ converges weakly in $H^1(D)$ and strongly in $L^2(D)$ to the unique solution of \eqref{eq:pp_subc}.
\end{enumerate}
\end{theorem}

\begin{remark}
Item two above corresponds to the critical setting in \cite{CioMur-2}, and item three is the sub-critical setting there. The first item is an analog of the super-critical setting for Stokes problem in \cite{Tartar,Allaire91-1,Allaire91-2}; in this setting, $\tilde u^\eps$ is of order $\sigma_\eps$ in $H^1$ and converges strongly to zero, and a proper rescaling of it has a limit given by an algebraic equation, just like the Darcy's law for the Stokes problems.

We emphasize that the coefficient $\mu_*$ in the strange term in the critical setting equals to the inverse of the multiplier, $\mathrm{Cap}^{-1}(T)$ or $1/2\pi$ depending on $d$, in the limiting algebraic equation in the super-critical setting with vanishing volume fraction of holes. The multiplier $c_*$ for positive volume fraction of holes is apparently different, and the relation between those multipliers is made clear in section \ref{sec:error}. This is an analog of a result by Allaire in \cite{Allaire-3}. 
\end{remark}

Theorem \ref{thm:homog} is proved in section \ref{sec:otest} by an adaptive oscillating test function method that relies on asymptotic behaviors of rescaled cell problems, which are studied in section \ref{sec:cell}.

\subsection{Correctors and Error estimates}

The second main theorem of the paper is the quantification of the homogenization results in the previous theorem. This is done only for the setting of vanishing volume fractions of holes.

\begin{theorem} \label{thm:error}
Assume {\upshape(P1)}, {\upshape(P2)}, $d\ge 3$ and that $a_\eps \ll \eps$. Let $\eta = a_\eps/\eps$ and set $v^\eps = \chi^\eta(\cdot/\eps\eta)$ with $\chi^\eta$ defined by the rescaled cell problem \eqref{eq:cell_2} in section \ref{sec:cell}. Then the following holds.
\begin{itemize}
	\item[(1)] If $\lim_{\eps \to 0} \sigma_\eps = 0$, assume further that $u = c_* g \in W^{2,d}_0(D)$. Then there exists $C$ depending only on $d, T$ and $D$ and we have
	\begin{equation*}
	\begin{aligned}
	&\|\nabla(\tilde u^\eps/\sigma^2_\eps - v^\eps g)\|_{L^2} \le C(\sigma_\eps+\eta^{\frac{d-2}{2}})\|u\|_{W^{2,d}(D)},\\
	&\|\tilde u^\eps/\sigma^2_\eps - v^\eps g\|_{L^2} \le C(\sigma_\eps^2 + \eps)\|u\|_{W^{2,d}(D)} \quad \text{and}\\
	&\|\tilde u^\eps/\sigma^2_\eps - c_* g\|_{L^2} \le C(\sigma_\eps^2 + \eps  + \eta^{\frac{d-2}{2}})\|u\|_{W^{2,d}(D)}.
	\end{aligned}
	\end{equation*}
	\item[(2)] If $\lim_{\eps\to 0} \sigma_\eps = 1$, let $u$ be the solution to \eqref{eq:pp_c} and assume further that $u \in W_0^{2,d}(D)$. Then there exists $C$ depending only on $d, T$ and $D$ such that
	\begin{equation*}
	\begin{aligned}
	&\|\tilde u^\eps - \mathrm{Cap}(T) v^\eps u\|_{H^1} \le C\eps\|u\|_{W^{2,d}(D)} \quad \text{and}\\
	&\|\tilde u^\eps - u\|_{L^2} \le C\eps\|u\|_{W^{2,d}(D)}.
	\end{aligned}
	\end{equation*}
	\item[(3)] If $\lim_{\eps\to 0} \sigma_\eps = \infty$, let $u$ be the solution to \eqref{eq:pp_subc} and assume further that $u \in W^{2,d}_0(D)$. Then there exists $C$ depending only on $d, T$ and $D$ such that
	\begin{equation*}
	\begin{aligned}
	&\|\tilde u^\eps - \mathrm{Cap}(T) v^\eps u\|_{H^1} \le C(\sigma_\eps^{-2} + \eta^{\frac{d-2}{2}})\|u\|_{W^{2,d}(D)} \quad \text{and}\\
	&\|\tilde u^\eps - u\|_{L^2} \le C(\sigma_\eps^{-2} + \eta^{\frac{d-2}{2}})\|u\|_{W^{2,d}(D)}.
	\end{aligned}
	\end{equation*}
	\end{itemize}
	For the setting of $d=2$, the above results hold with $\mathrm{Cap}(T)$ replaced by $2\pi$, and under the stronger condition that $u \in W^{2,\infty}_0(D)$.
\end{theorem}

\begin{remark} In each item above, the last inequality quantifies the error for the convergence of $\tilde u^\eps$; we can only hope to get convergence rate in $L^2$. The inequalities that precede the error estimates should be viewed as \emph{corrector} results. They provide necessary corrections that should be added to the homogenization limits to get stronger convergences and estimates. Take the first item, for instance, the corrector to the limit function $c_* f$ is given by $\phi^\eps := (v^\eps - c_*)g$, and with this corrector we have $\tilde u^\eps/\sigma_\eps^2 - (c_*g + \phi^\eps)$ strongly converges to zero in $H^1(D)$ with explicit error bounds.  

The above result in the critical setting was implied by the work of Kacimi and Murat \cite{Kacimi_Murat}; the results on the other settings are new as far as the author knows. In \cite{Kacimi_Murat}, such estimates were obtained by a quantitative general framework that strengthens those of \cite{CioMur-1}. In this paper, the results are proved by simple and natural quantification of the oscillating test function arguments. Moreover, the requirement that $u \in W_0^{2,d}$ in our results, for the $d\ge 3$ setting, seems an improvement from \cite{Kacimi_Murat,Allaire91-1}.
\end{remark}

\section{Asymptotic analysis for the cell problems}
\label{sec:cell}

We propose a unified approach to homogenization that is based on the standard energy method of Tartar with oscillating test functions built from cell problems. The new element is, when the hole-cell ratio vanishes in the limit, we need to consider cell problems with a parameter $\eta$ that stands for this ratio, and study the limit of the cell problem as $\eta \to 0$. This is the main objective of this section. Some earlier results can be found in \cite{Allaire-3}. 

\subsection{The cell problems} When the scaling factors for the holes and cells are the same, i.e. $a_\eps = \eps$, we have derived the cell problem \eqref{eq:cellprob} through formal two-scale expansions. For the setting of vanishing volume fractions, we introduce a parameter $\eta \in (0,1]$ which represents the ratio of the rescaling factors of the holes agains that of the cells, and consider the cell problem
\begin{equation}
\label{eq:cell_1}
\left\{
\begin{aligned}
&-\Delta_y \chi_\eta(y) = 1, &\quad &y \in Y_{f,\eta} = Q\setminus (\eta \ol T),\\
&\chi_\eta(y) = 0, &\quad &y \in \partial(\eta T),\\
&y \mapsto \chi_\eta(y) \text{ is } Q\text{-periodic}.
\end{aligned}
\right.
\end{equation}
Clearly, if $\eta = 1$, we come back to \eqref{eq:cellprob}. For each fixed $\eta$, the standard elliptic PDE theory yields a unique solution $\chi_\eta$ of the above problem. We aim to study the behavior of $\chi_\eta$ as $\eta \to 0$. 

For a reason that will be clear, we consider a rescaled version of the cell problem above. Dilate $Y_{f,\eta}$ by $1/\eta$, and define $\chi^\eta \in H^1_{\rm per}(\eta^{-1}Q)$ by
\begin{equation*}
\chi^\eta(x) = \eta^{d-2}\chi_\eta(\eta x), \qquad x \in \eta^{-1}Q \setminus \ol T,
\end{equation*}
and extend it by zero inside $T$. We also extend $\chi^\eta$ to $H^1_{\rm loc}(\R^d)$ periodically (with respect to $\eta^{-1}\Z^d$). 

\begin{remark}
\label{rem:periodic}
Throughout the paper, for a function $f$ in a cube $rQ$ satisfying periodic conditions at the boundary, we identify $f$ with its periodic extension in the whole space, and we also identify $rQ_\fp$ with a flat torus $r\R^d\setminus (r\Z^d)$ and view $f$ as a function on that torus, where the subscript $\fp$ indicates the identification of opposite sides of $\partial (rQ)$.
\end{remark}

The function $\chi^\eta$ is then the unique solution to the rescaled cell problem:
\begin{equation}
\label{eq:cell_2}
\left\{
\begin{aligned}
&-\Delta_x \chi^\eta(x) = \eta^d, &\quad &x \in \eta^{-1}Q\setminus T,\\
&\chi^\eta(x) = 0, &\quad &x \in \partial T,\\
&x \mapsto \chi^\eta(x) \text{ is } (\eta^{-1}Q)\text{-periodic}.
\end{aligned}
\right.
\end{equation}

\subsection{The scaling limit of cell problems}

In this subsection, we study the limit of $\chi^\eta$ as $\eta$ goes to zero. We restrict to the case of $d\ge 3$; the two dimensional setting is a bit different and is left in section \ref{sec:twod}. The advantage of treating $\chi^\eta$ instead of $\chi_\eta$ is the following uniform estimate.

\begin{lemma}
\label{eq:unif_bdd_u}
Assume {\upshape (P1)} and $d\ge 3$. Then there exists a constant $C$ that depends only on $T$ and $d$, such that
\begin{equation}
\label{eq:cell_C}
\|\nabla \chi^\eta\|_{L^2(\eta^{-1}Q)} \le C, \qquad c_\eta := \langle \chi^\eta \rangle_{\eta^{-1}Q} = \frac{1}{|\eta^{-1}Q|} \int_{\eta^{-1}Q} \chi^\eta \in (0,C].
\end{equation}
\end{lemma}

This lemma can be proved easily: we only needs to integrate the cell problems against $\chi^\eta$ and then use the Poincar\'e inequality \eqref{eq:poincare}. The sign of $c_\eta$ is due to the maximum principle which yields $\chi^\eta > 0$ in $\eta^{-1}Q\setminus \ol T$. With those uniform estimates, we can explore compactness method and prove the following key lemma which characterizes the asymptotic behavior of $\chi^\eta$.

\begin{lemma}\label{lem:celld3} Assume {\upshape (P1)} and $d \ge 3$. Let $w$ be the unique function in $\scD^{1,2}(\R^d)$ that satisfies
\begin{equation}
\label{eq:exterior}
\left\{
\begin{aligned}
&-\Delta w(x) = 0 &\qquad &x \in \R^d \setminus \ol T,\\
&w(x) = 1 &\qquad &x \in T.
\end{aligned}
\right.
\end{equation}
Then as $\eta \to 0$, $c_\eta$ converges to $c_* = (\mathrm{Cap}(T))^{-1}$; $\tilde \chi^\eta$ converges weakly in $L^p_{\rm loc}(\R^d)$, $p \in [1,\frac{2d}{d-2}]$, to $c_*(1-w)$, and $\nabla \tilde \chi^\eta$ converges weakly in $L^2_{\rm loc}(\R^d)$ to $-c_*\nabla w$.
\end{lemma}

In the statement of the theorem, $\scD^{1,2}(\R^d)$ is the so-called homogeneous Sobolev space and it is the completion of smooth and compactly supported functions with norm $\|\nabla\varphi\|_{L^2(\R^d)}$. We refer to \cite{DenLio} for more information regarding this space. \eqref{eq:exterior} is the exterior problem that determines the Newtonian capacity of $T$. In fact, by definition,
\begin{equation}
\label{eq:capdef}
\mathrm{Cap}(T) := \int_{\R^d\setminus \ol T} |\nabla w|^2.
\end{equation}
The weak formulation of problem \eqref{eq:exterior} reads
\begin{equation*}
\int_{\R^d\setminus \ol T} \nabla w \cdot \nabla \varphi + \int_{\partial T} \frac{\partial \varphi}{\partial \nu} = 0, \qquad \forall \varphi \in \scD^{1,2}(\R^d).
\end{equation*}
The requirement that $w \in \scD^{1,2}(\R^d)$ essentially imposes that $w$ decays to zero at infinity. In the proof of Lemma \ref{lem:celld3} below, we provide a proof for the existence and uniqueness of the solution to the exterior problem above$w$; another way to solve \eqref{eq:exterior} is by the potential theory.

\begin{proof} Observe that $\tilde \chi^\eta$ is in $H^1_{\rm loc}(\R^d)$, vanishes in $T$ and is $\eta^{-1}\Z^d$-periodic. We introduce some cut-off functions. Let $\rho: [0,\infty) \to [0,1]$ be smooth and satisfy $\rho \equiv 1$ in $[0,\frac12]$ and $\rho \equiv 0$ in $[1,\infty)$, with derivative $\rho' \in [-4,0]$. Then we define $\xi_\eta(x) = \rho(\eta |x|_\infty)$, where $|x|_\infty := \max_{i=1}^d |x^i|$ is the infinity norm of $x = (x^1,\cdots,x^d)$. Then there exists a constant $C$ depending only on $d$ and $\rho$, and
\begin{equation}
\label{eq:chi2star}
\xi_\eta \equiv 1 \text{ on } \frac{1}{2\eta}Q, \quad \|\nabla \xi_\eta\|_{L^\infty(\eta^{-1}Q)} \le C\eta, \quad\text{and}\quad \|\nabla\xi_\eta\|_{L^d(\eta^{-1}Q)} \le C.
\end{equation}

By the Sobolev embedding, we can find a constant $C_d$ independent of $\eta$, such that
\begin{equation}
\|\chi^\eta - c_\eta\|_{L^{\frac{2d}{d-2}}(\eta^{-1}Q)} \le C_{d}\|\nabla \chi^\eta\|_{L^2} \le C.
\end{equation}
We emphasize that for $L^{\frac{2d}{d-2}}$, the embedding above is scaling invariant and hence $C_d$ is independent of $\eta$. To summarize, we can find a $C$ still independent of $\eta$, and we have
\begin{equation*}
\|\nabla[(\chi^\eta-c_\eta) \xi_\eta]\|_{L^2(\R^d)} \le \|\xi_\eta \nabla \chi^\eta\|_{L^2(\eta^{-1}Q)} + \|\chi^\eta-c_\eta\|_{L^{\frac{2d}{d-2}}(\eta^{-1}Q)}\|\nabla \xi_\eta\|_{L^d(\eta^{-1}Q)} \le C.
\end{equation*}
This shows that $(\chi^\eta-c_\eta)\xi_\eta \in \scD^{1,2}(\R^d)$ and satisfies
\begin{equation*}
\|(\chi^\eta-c_\eta)\xi_\eta\|_{\scD^{1,2}(\R^d)} \le C.
\end{equation*}
By the Banach-Alaoglu theorem, there is a subsequence still denoted by $\eta \to 0$, a function $v \in \scD^{1,2}(\R^d)$ and a constant $\hat c \in (0,C]$, such that: $c_\eta \to \hat c$, $(\chi^\eta-c_\eta)\xi_\eta$ converges weakly in $\scD^{1,2}$ to $v$. We also have $\chi^\eta - c_\eta$  converges to $v$ weakly in $L^{\frac{2d}{d-2}}$ on any bounded set. 

We show next that the possible limit pair $(\hat c,v)$ is uniquely determined and, as a result, the whole sequence converges. To characterize the pair $(\hat c,v)$, we first pass to the limit in the weak formulation of \eqref{eq:cell_2}, and get
\begin{equation*}
-\Delta v = 0 \text{ in } \R^d\setminus \ol T, \quad v\rvert_{\partial T} = -\hat c.
\end{equation*}
in the sense that
\begin{equation*}
\int_{\R^d\setminus \ol T}  \nabla v \cdot \nabla \varphi +  \int_{\partial T} \varphi \frac{\partial v}{\partial \nu}= 0, \qquad \forall \varphi \in \scD^{1,2}(\R^d).
\end{equation*}
By taking the cut-off function $\xi_2 = \rho(2|x|_\infty))$ as the test function for \eqref{eq:cell_2}, we find
\begin{equation*}
\int_{\partial T} \frac{\partial v}{\partial \nu} \, d\sigma = \lim_{\eta \to 0} \int_{\partial T}\frac{\partial \chi^\eta}{\partial \nu}.
\end{equation*}
To compute the right hand side above, we integrate the first equation in \eqref{eq:cell_2} on both side and get
\begin{equation*}
\lim_{\eta \to 0} \int_{\partial T} \frac{\partial \chi^\eta}{\partial \nu} d\sigma = \lim_{\eta \to 0} \int_{\eta^{-1}Q \setminus \ol T} \eta^{d} \,dx = 1.
\end{equation*}
Let $w = -\frac{v}{\hat c}$, then $w \in H^1_{\rm loc}(\R^d)$  solves the problem \eqref{eq:exterior}.

On the other hand, the solution of \eqref{eq:exterior} must be unique. Indeed, if $w_1$ and $w_2$ are two solutions, then set $u = w_1 - w_2$; it satisfies
\begin{equation*}
-\Delta u = 0 \text{ in } \R^d\setminus T, \quad u \lvert_{\partial T} = 0.
\end{equation*}
Integrate the equation above against $u$ itself, we get
\begin{equation*}
\int_{\R^d\setminus T} |\nabla u |^2  = 0.
\end{equation*}
Since $u \in \scD^{1,2}(\R^d)$, we get $u = w_1 - w_2 = 0$ in $\R^d$.

To conclude, we proved that $c_\eta$ converges to $\hat c$, $\chi^\eta$ converges weakly in $L^{2d/(d-2)}_{\rm loc}$ to $\hat c(1-w)$, and $\nabla \chi^\eta$ converges weakly in $L^2_{\rm loc}$ to $-\hat c\nabla w$. Finally, $\hat c$ is determined by the identity
\begin{equation}
\int_{\R^d\setminus \ol T} |\nabla w|^2 = - \int_{\partial T} \frac{\partial w}{\partial \nu} = \frac1{\hat c} \int_{\partial T} \frac{\partial v}{\partial \nu} = \frac1{\hat c}.
\end{equation}
In view of the defining identity of capacity \eqref{eq:capdef}, we deduce that the limit of $c_\eta$ is indeed $c_*$ defined in \eqref{eq:chat}. The proof is now complete.
\end{proof}

\begin{remark} The proof above followed some ideas of Allaire outlined in \cite{Allaire-3} in the setting of Stokes problems. In section \ref{sec:otest}, we provide a quantitative analysis for this convergence result.
\end{remark}

Next we rescale $\chi^\eta$ further to obtain a function that oscillates in the same scale of $D^\eps$. We define
\begin{equation*}
v^\eps(x) := \chi^\eta\left(\frac{x}{\eps\eta}\right) = \eta^{d-2}\chi_\eta\left(\frac{x}{\eps}\right), \quad x \in \eps \R^d_{f,\eta}.
\end{equation*}
Then $v^\eps$ belongs to $H^1_{\rm loc}(\R^d)$, vanishes inside $\eps (k+\eta T)$, $k \in \Z^d$, and is $\eps \Z^d$-periodic. We check that $v^\eps$ is the unique solution to the rescaled cell problem:
\begin{equation}
\label{eq:cell_01}
\left\{
\begin{aligned}
& -\Delta v^\eps(x) = \frac{\eta^{d-2}}{\eps^2} = \frac{1}{\sigma_\eps^2}, &\qquad  &x \in \eps\R^d_{f,\eta},\\
&v^\eps(x) = 0, &\qquad &x \in \cup_{k\in \Z^d} \, \eps(k+\eta T).
\end{aligned}
\right.
\end{equation}
We will use $v^\eps$ to construct oscillating test functions in the proof of homogenization, and the following convergence results will be the key.

\begin{lemma}\label{lem:cell_1} Assume {\upshape(P1)} and $d \ge 3$. For each $\eps$ and $a_\eps$, let $\eta = a_\eps/\eps \in (0,1]$ and let $v^\eps$ be defined as above. Let $K$ be an open and bounded set in $\R^d$ with smooth boundary. Then there exists a positive constant $C$ depending only on $K$, $T$ and $d$, such that
\begin{equation}
\label{eq:gvbdd}
\|\nabla v^\eps\|_{L^2(K)} \le C\eps^{-1} \eta^{\frac{d-2}{2}} = C\sigma_\eps^{-1}.
\end{equation}
Moreover, let $c_*$ be defined as in \eqref{eq:chat}, then we have
\begin{equation}
\label{eq:cp_lim=c*}
\lim_{\eps \to 0} \|v^\eps - c_*\|_{L^2(K)} = 0.
\end{equation}
Finally, for the critical hole-cell ratio, i.e. $a_\eps = a^\eps_*$, we have
\begin{equation}
\label{eq:gvwl}
\nabla v^\eps \rightharpoonup 0 \quad \text{weakly in } L^2(K).
\end{equation}
\end{lemma}

\begin{remark} The estimate \eqref{eq:gvbdd} show that $\nabla v^\eps$ converges strongly to zero in the super-critical setting. For the critical ratio, \eqref{eq:gvbdd} says that $\nabla v^\eps \to 0$ holds but weakly. We remark also that, in the setting of vanishing volume fractions of hole, i.e. $a_\eps \ll \eps$, the convergence in \eqref{eq:cp_lim=c*} holds in $L^p$ for any $p \in [1,\frac{2d}{d-2}]$. This is a fact proved below.
\end{remark}

\begin{proof} We first prove the bound \eqref{eq:gvbdd}. By definition, we have the expression
\begin{equation*}
\nabla v^\eps(x) = \frac{1}{\eps \eta} \nabla \chi^\eta\left(\frac{x}{\eps \eta}\right),
\end{equation*}
The desired estimate can then be proved by breaking down the integral of $|\nabla v^\eps|^2$ into each $\eps$-square inside $K$, rescaling the $\eps$-squares to $\frac{1}{\eta}$-squares, and then applying \eqref{eq:cell_C}. 

Next, we prove \eqref{eq:cp_lim=c*}. In the setting of holes with positive volume fraction, i.e. $a_\eps = \eps$, it follows directly from the Riemann-Lebesgue lemma. We hence assume that the holes have vanishing volume fraction. The main idea is to use periodicity of $v^\eps$ and Lemma \ref{lem:celld3}. Let $\mathcal{I}_\eps = \mathcal{I}_\eps(K)$ denote the integer points in $k \in \Z^d$ such that $\eps (k+Q)$ has nonempty intersection with $K$. The cardinality of $\mathcal{I}_\eps$ is of order $\eps^{-d}$. Then, with $2^*$ representing the exponent $2d/(d-2)$, we compute
\begin{equation}
\label{eq:c_conv}
\begin{aligned}
\|v^\eps - c_*\|_{L^{2^*}(K)}^{2^*} &= \sum_{k \in \mathcal{I}_\eps} \int_{\eps (k+Q)} \left| \chi^\eta(\textstyle \frac{x}{\eps \eta}) - c_*\right|^{2^*}dx\\
&= \sum_{k \in \mathcal{I}_\eps} \eps^d \eta^d \int_{\frac{1}{\eta}(k+Q)} \left| \chi^\eta(y) - c_*\right|^{2^*}dx\\
&\le \sum_{k \in \mathcal{I}_\eps} C\eps^d \left(|c_\eta - c_*|^{2^*} + \eta^d \int_{\frac{1}{\eta}Q} \left| \chi^\eta(y) - c_\eta\right|^{2^*}dx\right)\\
&\le C\left(|c_\eta - c_*|^{2^*}+\eta^d\right).
\end{aligned}
\end{equation}
In the third line, we used H\"oler inequality and the constant depends only on $d$, and in the last line, we used \eqref{eq:chi2star}. In view of $\eta = a_\eps/\eps \to 0$ and Lemma \ref{lem:celld3}, we conclude that the right hand side above vanishes in the limit.

\medskip

Now we prove \eqref{eq:gvwl}. This is essentially the Riemann-Lebesgue lemma, although there is another parameter $\eta_\eps$ that goes to zero with $\eps$. We note that $\|\nabla v^\eps\|_{L^2(K)}$ is bounded in view of \eqref{eq:gvbdd}, since $a_\eps = a^\eps_*$ implies $\sigma_\eps = 1$. Fix an arbitrary smooth test function $\psi $ compactly supported in $D$. Then
\begin{equation*}
\begin{aligned}
\int_D \partial_j v^\eps \psi &= \sum_{k \in \mathcal{I}_\eps} \int_{\eps(k+ Q)} \frac{1}{\eps \eta} (\partial_j \chi^\eta)\left(\frac{x}{\eps \eta}\right) \psi(x) dx\\
&= \sum_{k \in \mathcal{I}_\eps} \int_{\eta^{-1}(k+Q)} (\eps \eta)^{d-1} (\partial_j \chi^\eta)(y) \psi(\eps k + \eps \eta y) dy\\
&= \sum_{k \in \mathcal{I}_\eps} \left(\psi(\eps k) \int_{\eta^{-1}(k+Q)} (\eps \eta)^{d-1} (\partial_j \chi^\eta)(y) dy + r_\eps(k)\right).
\end{aligned}
\end{equation*}
In the last line we add and subtract $\psi(\eps k)$ in each $\eps$-cube. Then the first term vanishes due to periodicity. For the remainder $r_\eps(k)$, we use the mean value theorem on $\psi$ and get
\begin{equation*}
\begin{aligned}
|r_\eps(k)| &= \left|\int_{\eta^{-1}(k+Q)}  (\eps\eta)^{d-1} \partial_j\chi^\eta(y) [\psi(\eps k + \eps \eta y) - \psi(\eps k)] dy \right|\\
&\le (\eps \eta)^d \|D\psi\|_{L^\infty} \|\nabla \chi^\eta\|_{L^2(\eta^{-1}(k+Q))}|\eta^{-1}(k+Q)|^{\frac12}\\
&\le C\eps^d \eta^{\frac d2} \|D\psi\|_{L^\infty}.
\end{aligned}
\end{equation*}
Note that the bound above is uniform in $k$. Plugging this into the summation above, we get
\begin{equation*}
	\lim_{\eps \to 0} \int_D \partial_j v^\eps \psi \, dx = 0.
\end{equation*}
This completes the proof of the theorem.
\end{proof}


\section{A unified proof for periodic homogenizations}
\label{sec:otest}

In this section, we prove Theorem \ref{thm:homog} by a unified approach based on the oscillating test function method. The oscillatory test functions are constructed using $v^\eps$ of last section. Again, we focus on the case of $d\ge 3$ first.

Here is an outline of the method: the starting point is to conclude from the uniform estimates \eqref{eq:ubdd} that $\nabla \tilde u^\eps/(1\wedge\sigma_\eps)$ and $\tilde u^\eps/(1\wedge \sigma^2_\eps)$ converges weakly in $L^2(D)$ and in $H^1(D)$ respectively. Then homogenization is proved by showing that the possible limits are uniquely determined.

To determine the limits, fix an arbitrary smooth function $\varphi \in C^\infty(D)$ with compact support in $D$, consider the oscillating function $\varphi v^\eps$ which belongs to $H^1_0(D^\eps)$ and test it against the equation of $u^\eps$, i.e. \eqref{eq:ppeps}. We then get
\begin{equation*}
\int_D \varphi \nabla \tilde u^\eps \cdot \nabla v^\eps  + \int_D v^\eps \nabla \tilde u^\eps \cdot \nabla \varphi = \int_D f \varphi v^\eps.
\end{equation*}
On the other hand, the oscillatory function $\varphi \tilde u^\eps$ also belongs to $H^1_0(D)$ and we test it against the equation of $v^\eps$, that is \eqref{eq:cell_01}. Then we get
\begin{equation*}
\int_D \varphi \nabla v^\eps \cdot \nabla \tilde u^\eps + \int_D \tilde u^\eps \nabla v^\eps \cdot \nabla \varphi = \frac{1}{\sigma_\eps^2} \int_D \tilde u^\eps \varphi.
\end{equation*}
Subtracting the two identities above, one obtains
\begin{equation}
\label{eq:uv}
\int_D v^\eps \nabla \tilde u^\eps \cdot \nabla \varphi  - \int_D \tilde u^\eps \nabla v^\eps \cdot \nabla \varphi  = \int_D   \varphi f v^\eps - \int_D \varphi \frac{\tilde u^\eps}{\sigma_\eps^2}. 
\end{equation}

We then pass the limit $\eps \to 0$ in \eqref{eq:uv}, and characterize the limits of each terms according to different regimes of $a_\eps/\eps$. Let us number the integrals there by $I_1, I_2$, $I_3$ and $I_4$ according to the order of their appearances. We study them case by case for all regimes of hole-cell ratios.

\subsection{The case of super-critical hole-cell ratio} In this setting, $\sigma_\eps \to 0$. We assume that through a subsequence still denoted by $\tilde u^\eps$, we have $\tilde u^\eps/\sigma^2_\eps$ converges weakly to $u \in L^2(D)$. To pass to the limit in \eqref{eq:uv} and identify $u$, we distinguish two two subcases: positive volume fraction of holes (for simplicity $a_\eps = \eps$) and vanishing volume fraction ($a_\eps \ll \eps$).

\begin{proof}[Proof of part (1) in Theorem \ref{thm:homog}] Consider the case $a_\eps = \eps$. We have $\eta = 1$, $\sigma_\eps = \eps$ and $v^\eps$ is $\eps$-periodic. This is the classical homogenization setting and we have recalled the formal two-scale expansion argument in the Introduction. Here we give the rigorous proof. In view of Lemma \ref{lem:cell_1} and the estimate \eqref{eq:ubdd}, we have
\begin{equation*}
|I_1| \le \|v^\eps\|_{L^2(D)}\|\nabla \tilde u^\eps\|_{L^2(D)} \|\nabla \varphi\|_{L^\infty} \le C\eps \|\nabla \varphi\|_{L^\infty},
\end{equation*}
which converges to zero as $\eps \to 0$; similarly, we have
\begin{equation*}
|I_2| \le \|\tilde u^\eps\|_{L^2(D)} \|\nabla v^\eps\|_{L^2(D)} \|\nabla \varphi\|_{L^\infty} \le C\eps^2 \eps^{-1}\|\nabla \varphi\|_{L^\infty} ,
\end{equation*}
which converges to zero as well. For $I_3$, by the Riemann-Lebesgue lemma we have
\begin{equation}
\lim_{\eps \to 0} I_3 = \langle \chi \rangle_{Q} \int_D f\varphi = c_*\int_D f\varphi,
\end{equation}
where $\chi$ is the unique solution of the standard cell problem \eqref{eq:cellprob}, and we have used the definition $c_* = \langle \chi \rangle_Q$ of this setting. Finally, since $\tilde u^\eps/\sigma_\eps^2$ converges weakly in $L^2$, we have
\begin{equation*}
 \lim_{\eps \to 0} I_4 = \int_D \varphi u.
 \end{equation*} 
Hence, we showed that the limit of \eqref{eq:uv} reads
\begin{equation}
\label{eq:supc}
\int_D (c_*f - u) \varphi = 0.
\end{equation}
Since $\varphi$ is arbitrary, we conclude that $u = c_* f$.

\medskip

Consider the case $a_\eps \ll \eps$. Then the supercritical condition $\lim_{\eps \to 0} \sigma_\eps = 0$ implies that $a^\eps_* \ll a_\eps$. By Lemma \ref{lem:cell_1} and the remark there, we know $v^\eps$ is uniformly bounded in $L^{2^*}(D)$, where $2^* = \frac{2d}{d-2}$. Hence,
\begin{equation*}
|I_1| \le \|v^\eps\|_{L^{2^*}(D)}\|\nabla \tilde u^\eps\|_{L^2(D)} \|\nabla \varphi\|_{L^d(D)} \le C\sigma_\eps \|\nabla \varphi\|_{L^d},
\end{equation*}
On the other hand, by \eqref{eq:ubdd} and by \eqref{eq:gvbdd}, we have
\begin{equation*}
|I_2| \le \|\tilde u^\eps\|_{L^2(D)} \|\nabla v^\eps\|_{L^2(D)} \|\nabla \varphi\|_{L^\infty(D)} \le C\sigma_\eps\|\nabla \varphi\|_{L^\infty}.
\end{equation*}
Both terms vanishes in the limit. For $I_3$, we apply Lemma \ref{lem:cell_1}, in particular the fact that $v^\eps$ converges strongly in $L^{2^*}(D)$ (hence also in $L^2(D)$) to $c_* = 1/\mathrm{Cap}(T)$. This shows: 
\begin{equation}
\lim_{\eps \to 0} I_3 = \frac{1}{\mathrm{Cap}(T)} \int_D f\varphi.
\end{equation}
Convergence of $I_4$ is straightforward as before. Then \eqref{eq:supc} follows again and we get $u = c_*f$ with $c_* = 1/\mathrm{Cap}(T)$. This completes the proof of the first part of Theorem \ref{thm:homog}.
\end{proof}

\medskip

\subsection{The case of critical hole-cell ratio} Here $a_\eps$ is critical and comparable to $a^\eps_*$. For simplicity, assume $\lim_{\eps \to 0} \sigma_\eps = 1$. In view of \eqref{eq:ubdd}, we assume that along a subsequence still denoted by $\tilde u^\eps$, the function $\tilde u^\eps$ converges weakly in $H^1$ to $u \in H^1_0(D)$. Then $\nabla \tilde u^\eps$ converges weakly in $L^2$ to $\nabla u$. 

\begin{proof}[Proof of part (2) in Theorem \ref{thm:homog}] It remains to study the limit of \eqref{eq:uv} to identify $u$. For $I_1$, we use the strong convergence \eqref{eq:cp_lim=c*} and the weak convergence of $\nabla \tilde u^\eps$ to conclude that
\begin{equation*}
\lim_{\eps \to 0} I_1 = \lim_{\eps \to 0} \, c_* \int_D \nabla \tilde u^\eps \cdot \nabla \varphi = c_* \int_D \nabla u \cdot \nabla \varphi,
\end{equation*}
where $c_*$ is $1/\mathrm{Cap}(T)$ according to \eqref{eq:chat}. For $I_2$, we use the weak convergence of $\nabla v^\eps$ in \eqref{eq:gvwl} and the strong convergence of $\tilde u^\eps$, and get
\begin{equation*}
\lim_{\eps \to 0} I_2 = \lim_{\eps \to 0} \int_D \nabla v^\eps \cdot \nabla \varphi \, u = 0.
\end{equation*}
The terms $I_3$ and $I_4$ are trivial in this setting, and we have
\begin{equation*}
\lim_{\eps \to 0} I_3 = \int_D c_* f\varphi, \qquad \lim_{\eps \to 0} I_4 = \int_D fu.
\end{equation*}
Hence, passing to the limit $\eps \to 0$ on \eqref{eq:uv}, we obtain:
\begin{equation}
\int_D c_*\left(\nabla u \cdot \nabla \varphi - f\varphi + \frac{1}{c_*} u \varphi\right) = 0.
\end{equation}
Note that we used the positivity of $c_*$ which is clear. 

Finally, since $\varphi$ is arbitrary, the above implies that the limit function $u \in H^1_0(D)$ is a solution to 
\begin{equation*}
-\Delta u + \frac1{c_*} u = f \quad \text{in } D.
\end{equation*}
Note $1/c_* = \mathrm{Cap}(T)$ is positive. The solution to this Dirichlet problem is unique. This completes the proof of the second item of Theorem \ref{thm:homog}.
\end{proof}

\medskip

\subsection{The case of sub-critical hole-cell ratios} Now we consider the remaining situation where $a_\eps \ll a^\eps_*$ or $\lim_{\eps \to 0} \sigma_\eps = \infty$. Again, by the uniform estimates \eqref{eq:ubdd} we take a subsequence still denoted by $\tilde u^\eps$ for which $\tilde u^\eps$ converges weakly in $H^1$ and strongly in $L^2$ to $u \in H^1_0(D)$.

\begin{proof}[Proof of part (3) in Theorem \ref{thm:homog}] We study the limit of \eqref{eq:uv} to identify the limit $u$. The terms $I_1$ and $I_3$ can be treated exactly as in the case of critical hole-cell size ratio. The limits there still hold. The limit of $I_4$ vanishes since $\tilde u^\eps$ is bounded while $\sigma_\eps \to \infty$. It remains to study $I_2$. We use \eqref{eq:gvbdd} to conclude that
\begin{equation*}
|I_2| \le \|\nabla v^\eps\|_{L^2}\|\tilde u^\eps\|_{L^2} \|\nabla \varphi\|_{L^\infty} \le C\sigma_\eps^{-1} \to 0.
\end{equation*}
Combine those results; we see that \eqref{eq:uv} becomes
\begin{equation}
\int_D c_*\left(\nabla u \cdot \nabla \varphi - f\varphi\right) = 0
\end{equation}
in the limit. By the positivity of $c_*$ and the arbitrariness of $\varphi$, we conclude that the only possible limit $u \in H^1_0(D)$ is the unique solution to the problem \eqref{eq:pp_subc}. This completes the proof of item three of Theorem \ref{thm:homog}.
\end{proof}

\section{On the correctors and error estimates}
\label{sec:error}

In this section, we quantify the unified proof of the last section to get the error estimates in \ref{thm:homog}, for the cases when the holes have vanishing volume fractions. It is remarkable that this quantification is quite easy to do using tools from potential theory. 

It is almost clear from the proof of Theorem \ref{thm:homog} that, to quantify the homogenization error, we need to find the convergence rate of $v^\eps \to c_*$ in $L^2$. This is done in the next key lemma, for $d\ge 3$.

\begin{lemma}\label{lem:rvtc} Assume {\upshape (P1)} and $d \ge 3$. For each $a_\eps$ and $\eps$, let $\eta = a_\eps/\eps$, let $v^\eps$ be defined as in \eqref{eq:cell_01} and let $c_*$ be defined as in \eqref{eq:chat}. Let $K$ be an open and bounded set in $\R^d$ with smooth boundary. Then there exists a positive constant $C$ depending only on $K$, $T$ and $d$, such that
\begin{equation*}
\|v^\eps - c_*\|_{L^{2^*}(K)} \le C\eta^{\frac{d-2}{2}}.
\end{equation*}
\end{lemma}

Our proof below relies heavily on the potential theory for the Laplace operator. We refer the readers to the Appendix and the references therein for more details.

\begin{proof} {\itshape Step 1: The convergence rate of $c_\eta = \fint_{\eta^{-1}Q} \tilde\chi^\eta \to c_*$}. Here $\chi^\eta$ is the solution to \eqref{eq:cell_1}. Let $G^\eta$ be the Green's function of the Laplace operator in the flat torus $\eta^{-1}Q_\fp$; see \eqref{eq:Geta}. Our plan is to explore the explicit formula of $\chi^\eta$ given by the potential theory. We refer to Appendix \ref{app:lp} for the terminologies and notations. Let $\chi^\eta = G^\eta + \phi^\eta$. Then by definition $\phi^\eta$ satisfies
\begin{equation*}
-\Delta \phi^\eta = 0 \quad\text{on } \frac{1}{\eta} Q_{\fp}\setminus \ol T, \qquad \phi^\eta\vert_{\partial T} = g^\eta := -G^\eta\vert_{\partial T}.
\end{equation*}
In view of \eqref{eq:Geta_d}, $g^\eta$ is smooth on $\partial T$. By Theorem, $\mathrm{Ran}(-\frac12 I + \cK^{(\eta)}_{T,\fp})$ is the whole space and, hence, $\phi^\eta$ can be represented by a double layer potential. To find the structure of the potential, we decompose $g^\eta$ according to Corollary \ref{cor:Gam}: Let $X$ denote the subspace $\mathrm{Ran}(-\frac12 I + \cK_T)$ of $L^2(\partial T)$,then
\begin{equation*}
g^\eta = g^\eta_1 + r^\eta, \quad g^\eta_1 \in X \quad \text{and} \quad r^\eta = \langle g^\eta, \varphi_*\rangle.
\end{equation*}
As a result, $\phi^\eta$ can be represented by
\begin{equation}
\phi^\eta = \cD^{(\eta)}_{T,\fp} [\psi^\eta_1] + r^\eta, \quad\text{where} \quad \psi^\eta_1 = (-\frac12 I+ \cK^{\eta})^{-1} g_1^\eta.
\end{equation}

\medskip

Back to the function $\chi^\eta$, its average in $\frac1\eta Q$ can be written as
\begin{equation}
\label{eq:r_av}
c_\eta = r^\eta + \langle G^\eta \rangle_{\frac1\eta Q} + \langle \cD^{(\eta)}_{T,\fp} [\psi_1^\eta] \rangle_{\frac1\eta Q}.
\end{equation}
We estimates those numbers on the right hand side one by one. By Corollary \ref{cor:Gam}, we have
\begin{equation*}
\begin{aligned}
r^\eta &= -\langle G^\eta, \varphi_* \rangle_{L^2(\partial T), L^2(\partial T)} \\
&= -\langle \Gamma\big\rvert_{\partial T}, \varphi_* \rangle - \eta^{d-2} \langle R(\eta\cdot),\varphi \rangle = \frac{1}{\mathrm{Cap}(T)} - \eta^{d-2} \int_{\partial T} R(\eta x) \varphi_*(x).
\end{aligned}
\end{equation*}
Here, $R$ is the defined in the formula \eqref{eq:Geta_d}. This leads to
\begin{equation}
\label{eq:crate-1}
|r^\eta - c_*| \le \eta^{d-2}\|R\|_{L^2(\partial T)}.
\end{equation}

Secondly, for the average of $G^\eta$, we calculate
\begin{equation*}
\left| \frac{1}{|\eta^{-1} Q|} \int_{\frac1\eta Q \setminus \ol T} \Gamma(x) dx  \right| \le \eta^d \int_{B_{\frac1{2\eta}}} \Gamma + \eta^d \int_{\frac1\eta(Q\setminus \ol B_{1/2})} \Gamma \le C_d\,\eta^{d-2}. %
\end{equation*}
Then since $G^\eta(x) = \Gamma(x) + \eta^{d-2}R(\eta x)$ for all $x \in \frac1\eta Q\setminus \ol T$, we have
\begin{equation}
\label{eq:crate-2}
\left|\langle G^\eta \rangle_{\frac1\eta Q} \right| \le \langle \Gamma \rangle_{\frac1\eta Q} + \eta^{d-2}\|R\|_{L^\infty} \le C\eta^{d-2},
\end{equation}
where $C$ depends only on $d$ and $T$.

For the third term in the right hand side of \eqref{eq:r_av}, we note
\begin{equation*}
\cD^{(\eta)}_{T,\fp}[\psi_1^\eta] = \eta^d \int_{\partial T} \nu_y \cdot \nabla\Gamma(x-y) \psi^\eta_1(y) d\sigma + \eta^{d-1}\int_{\partial T} \nu_y \cdot (\nabla R)(\eta(x-y)) \psi^\eta_1(y) d\sigma.
\end{equation*}
Using the explicit form of $\nabla \Gamma$ and by the same method as above, we can average in $x$ in the first integral and estimate the result. For the second integral above, we note that $\|\nabla R\|_{L^\infty}$ is finite and depends only on $T$. We hence get
\begin{equation}
\label{eq:crate-3}
|\langle \cD^{(\eta)}_{T,\fp}[\psi_1^\eta]\rangle_{\eta^{-1}Q}| \le C\eta^{d-1}(1+\|\nabla R\|_{L^\infty})\|\psi^\eta_1\|_{L^2(\partial T)},
\end{equation}
for some constant $C$ that depends only on $d$ and $T$. We conclude that the rate of $c_\eta \to c_*$ can be characterized once we get a quantitative estimate of $\|\psi^\eta_1\|_{L^2(\partial T)}$ in terms of $\eta$. 

\smallskip

\noindent{\itshape Step 2: The estimate of $\|\psi^\eta_1\|_{L^2(\partial T)}$}. By the decomposition \eqref{eq:NPop-d}, the function $\psi^\eta_1$ is the unique solution to the equation
\begin{equation*}
\left(-\frac12 I + \cK_T + \cR^\eta \right)[\psi^\eta_1] = g^\eta_1.
\end{equation*}
where $\cR^\eta = \eta^{d-1}\cR_2$ and $\cR_2$ is defined in \eqref{eq:Rijdef}. We seek a solution of the form $\psi^\eta_1 = \xi_1 + \theta$ where $\xi_1 \in L^2_0(\partial T)$ and $\theta \in \R$. Plug this ansatz into the equation above, we find
\begin{equation}
\label{eq:psi1}
(-\frac12 I + \cK_T)[\xi_1] + \cR^\eta[\xi_1] -\eta^d|T|\theta = g^\eta_1.
\end{equation}
Project the above equation to $X$, and let $\Pi_X$ denote this projection operator. Recall that $g^\eta_1$ is in the range, we get
\begin{equation*}
(-\frac12 I + \cK_T)[\xi_1] + \Pi_X \cR^\eta[\xi_1] = g^\eta_1.
\end{equation*}
We know that $-\frac12 I + \cK_T$ is invertible from $L^2_0(\partial T)$ to $X$ with bounded operator norm; see Theorem \ref{thm:layerp}. On the other hand, $\|\cR^\eta\|_{L^2\to L^2}$ is of order $\eta^{d-1}$, and $\|\Pi_X\cR^\eta\|_{L^2_0\to X}$ is of the same order. By standard perturbation theory, see e.g. \cite[Theorem IV.1.16]{Kato}, for $\eta$ sufficiently small, the operator norm $\|(-\frac12 I + \cK_T + \Pi_X \cR^\eta\|_{L^2_0\to X}$ can be bounded by $C$ independent of $\eta$. We hence get
\begin{equation}
\label{eq:psi1-1}
\|\xi_1\|_{L^2(\partial T)} \le C\|g^\eta_1\|_{L^2(\partial T)} \le C.
\end{equation}

Now the projection of \eqref{eq:psi1} to the kernel of $-\frac12 I + \cK_T$ reads
\begin{equation*}
\eta^d\theta|T| = \langle \cR^\eta[\xi_1], \varphi_*\rangle = \langle \xi_1, (\cR^\eta)^* \varphi_* \rangle
\end{equation*}
We then get the estimate
\begin{equation}
\label{eq:psi1-2}
|\theta||T| \le \frac{1}{\eta^d}\|\xi_1\|_{L^2(\partial T)} \|(\cR^{\eta})^* \varphi_*\|_{L^2} \le \frac{1}{\eta}\|\xi_1\|_{L^2(\partial T)}\|\varphi_*\|_{L^2(\partial T)} \le C\eta^{-1}.
\end{equation}
Here we used the fact that $\|(\cR^\eta)^*\|_{L^2\to L^2} \le C\eta^{d-1}\|\nabla R\|_{L^\infty}$. Combine \eqref{eq:psi1-1} with \eqref{eq:psi1-2} we get $\|\psi^\eta_1\|_{L^2(\partial T)} \le C\eta^{-1}$. Using this estimate in \eqref{eq:crate-3} we get
\begin{equation*}
\left|\langle \cD^{(\eta)}_{T,\fp}[\psi^\eta_1] \rangle_{\eta^{-1}Q} \right| \le C\eta^{d-2}.
\end{equation*}
Combine this estimate with \eqref{eq:crate-1} and \eqref{eq:crate-2}, we finally obtain, for $d\ge 3$,
\begin{equation}
\label{eq:crate-4}
|c_\eta - c_*| \le C\eta^{d-2},
\end{equation}
where the constant $C$ depends only on $d$ and $T$.

\smallskip

\noindent{\itshape Step 3: Convergence rate for $\|v^\eps - c_*\|_{L^{2^*}(K)} \to 0$.} We examine the proof of Lemma \ref{lem:cell_1}, and apply \eqref{eq:crate-4} in the last line of \eqref{eq:c_conv}. We get
\begin{equation*}
\|v^\eps - c_*\|_{L^{2^*}(K)}^{2^{*}} \le C\left((\eta^{(d-2)/2})^{2^*} + \eta^d\right) \le C\eta^d.
\end{equation*}
where the constant $C$ depends only on $d$, $T$ and $K$. This completes the proof of Lemma \ref{lem:rvtc}.
\end{proof}
\subsection{Proof of Theorem \ref{thm:error}}

Now we are ready to prove the corrector and error estimates in Theorem \ref{thm:error}. We emphasize that those results are only for the setting with vanishing volume fraction of holes, i.e. $a_\eps \ll \eps$.

\begin{proof}[Proof of Theorem \ref{thm:error}] {\itshape Case 1: super-critical hole-cell ratios.} Here, $a^\eps_* \ll a_\eps \ll \eps$ and $\lim_{\eps \to 0} \sigma_\eps = 0$ hold. Moreover, the assumption that $u = \langle \chi \rangle_Q g \in W^{2,d}_0(D)$ implies that the source term in \eqref{eq:ppeps} satisfies $g \in W^{2,d}_0(D)$. 

Let $\eta = a_\eps/\eps$ and let $v^\eps$ be defined as in \eqref{eq:cell_01}. Define $\xi^\eps := u^\eps/\sigma^2_\eps - v^\eps g$; then $\xi^\eps \in H^1_0(D^\eps)$. Direct computation shows that
\begin{equation*}
-\Delta \xi^\eps = 
v^\eps \Delta g + 2\nabla v^\eps \cdot \nabla g \qquad \text{in } D^\eps.
\end{equation*}
Integrate the above equation against $\xi^\eps$; we get
\begin{equation*}
\int_{D^\eps} |\nabla \xi^\eps|^2 = \int_{D^\eps} \xi^\eps v^\eps \Delta g + 2\int_{D^\eps} \xi^\eps\nabla(v^\eps - c_*)\cdot \nabla g.
\end{equation*}
Note that in the last integral above, we inserted $-c_*$ in a gradient term which does not change the integral. We estimate
\begin{equation}
\label{eq:ee_1}
\begin{aligned}
\left| \int_{D^\eps} \xi^\eps v^\eps \Delta g \right| &\le \|\xi^\eps\|_{L^2(D^\eps)} \|v^\eps\|_{L^{2^*}(D^\eps)} \|\Delta g\|_{L^d(D)}\\
&\le \sigma_\eps \|\nabla\xi^\eps\|_{L^2(D^\eps)} \|v^\eps\|_{L^{2^*}(D^\eps)}\|\Delta g\|_{L^d(D)} .
\end{aligned}
\end{equation}
In the first line, we applied H\"older inequality and noted that $v^\eps$ in bounded in $L^{2^*}(D^\eps)$, as indicated by \eqref{eq:cp_lim=c*} and the remark there. In the second line, we used the Poincar\'e inequality. For the other integral, we compute
\begin{equation*}
\int_{D^\eps} \xi^\eps \nabla(v^\eps - c_*)\cdot \nabla g = \int_{D^\eps} \nabla \cdot \left[\xi^\eps(v^\eps-c_*)\nabla g\right] - (v^\eps - c_*) \nabla \xi^\eps\cdot \nabla g - \xi^\eps(v^\eps-c_*)\Delta g.
\end{equation*}
We view the right hand side as three integrals. Then the first one is zero since $\xi^\eps$ vanishes on $\partial D^\eps$. The third integral can be treated as in \eqref{eq:ee_1}. To estimate the second one, we have
\begin{equation*}
\left| \int_{D^\eps} (v^\eps - c_*) \nabla \xi^\eps\cdot \nabla g \right| \le \|v^\eps - c_*\|_{L^{2^*}(D^\eps)}\|\nabla \xi^\eps\|_{L^2(D^\eps)} \|\nabla g\|_{L^d(D)}.
\end{equation*}
Combining the estimates above and apply Lemma \eqref{lem:rvtc}, we get
\begin{equation*}
\|\nabla\tilde \xi^\eps\|_{L^2(D)} \le C\|g\|_{W^{2,d}(D)}[\sigma_\eps + \|v^\eps-c_*\|_{L^2(D^\eps)}] \le C(\sigma_\eps + \eta^{\frac{d-2}{2}})\|g\|_{W^{2,d}(D)}.
\end{equation*}
By the Poincar\'e inequality, we also have
\begin{equation*}
\|\tilde \xi^\eps\|_{L^2(D)} \le C(\sigma_\eps^2 + \eps)\|g\|_{W^{2,q}(D)}.
\end{equation*}
Here, we used the relation that $\eta^{\frac{d-2}{2}} = \eps/\sigma_\eps$. We hence proved the first two estimates in part (1) of Theorem \ref{thm:error}. The third inequality follows from the relation that
\begin{equation*}
\frac{\tilde u^\eps}{\sigma_\eps^2} - c_*g = \frac{\tilde u^\eps}{\sigma_\eps^2} - v^\eps g + (v^\eps - c_*)g,
\end{equation*}
the result of Lemma \ref{lem:rvtc} and the triangle inequality.

\medskip

{\itshape Case 2: Critical hole-cell ratio.} Now we consider the case of $a_\eps \sim a^\eps_*$ which also implies that $\sigma_\eps \sim 1$. For simplicity, we assume that $\sigma_\eps = 1$ for the sequence of $\eps$ that converges to $0$. Then $\eta^{\frac{d-2}{2}} = \eps$.

Similar to the previous case, let $\xi^\eps = u^\eps - c_*^{-1} v^\eps u$. Then $\xi^\eps \in H^1_0(D^\eps)$ and it solves the equation
\begin{equation*}
-\Delta \xi^\eps = \frac1{c_*}(v^\eps - c_*) \Delta u + \frac{2}{c_*}\nabla v^\eps\cdot \nabla u \qquad \text{in } D^\eps.
\end{equation*}
Integrate against $\xi^\eps$, we have
\begin{equation*}
\begin{aligned}
\int_{D^\eps} |\nabla \xi^\eps|^2 &= \frac{1}{c_*}\int_{D^\eps} (v^\eps-c_*)\xi^\eps \Delta u + \frac{2}{c_*}\int_{D^\eps} \xi^\eps \nabla (v^\eps- c_*)\cdot \nabla u\\
&= -\frac{1}{c_*} \int_{D^\eps} (v^\eps-c_*)\xi^\eps \Delta u - \frac{2}{c_*} \int_{D^\eps} (v^\eps-c_*) \nabla \xi^\eps \cdot \nabla u.
\end{aligned}
\end{equation*}
We then note that $\tilde \xi \in H^1_0(D)$, and by Poincar\'e inequality we get
\begin{equation*}
\|\nabla \tilde \xi^\eps\|^2_{L^2(D)} \le C\|v^\eps-c_*\|_{L^{2^*}(D^\eps)}\|\nabla \tilde\xi^\eps\|_{L^2} (\|\Delta u\|_{L^d(D)} + \|\nabla u\|_{L^d(D)}).
\end{equation*}
By Lemma \ref{lem:rvtc} and by the Poincar\'e inequality again we finally get
\begin{equation*}
\|\xi^\eps\|_{H^1(D)} \le C\eta^{\frac{d-2}{2}}\|u\|_{W^{2,d}(D)} = C\eps\|u\|_{W^{2,d}(D)}.
\end{equation*}
This finishes the proof of the first estimate in item two of Theorem \ref{thm:error}; the other estimates follows, again, by adding the correction $(c_*^{-1}v^\eps - 1)u$ to $u$ and by the triangle inequality.

\medskip

{\itshape Case 3: Sub-critical hole-cell ratios.} Finally we consider the case of $a_\eps \ll a^\eps_*$ which implies that $\lim_{\eps\to 0} \sigma_\eps = \infty$. We still let $\xi^\eps$ denote the distance function $u^\eps - c_*^{-1}v^\eps u$. Then
\begin{equation*}
-\Delta \xi^\eps = -\frac{1}{c_*} \frac{u}{\sigma_\eps^2} + \frac1{c_*}(v^\eps - c_*) \Delta u + \frac{2}{c_*}\nabla v^\eps\cdot \nabla u \qquad \text{in } D^\eps
\end{equation*}
By exactly the same arguments of the previous cases and the usual Poincar\'e inequality $\|\tilde\xi^\eps\|_{L^2(K)} \le C(K)\|\nabla \xi^\eps\|_{L^2(K)}$, we get
\begin{equation*}
\begin{aligned}
\|\nabla \tilde \xi^\eps\|_{L^2(D)} &\le C(\sigma_\eps^{-2}\|u\|_{L^2(D)} +\|u\|_{W^{2,\infty}(D)}\|v^\eps-c_*\|_{L^2(D)})\\
&\le C(\sigma_\eps^{-2} + \eta^{\frac{d-2}{2}})\|u\|_{W^{2,d}(D)}.
\end{aligned}
\end{equation*}
This finish the proof of the first inequality in item three of Theorem \ref{thm:error}; the other one can be treated as before.
\end{proof}

\medskip


\medskip

\section{The two dimensional setting}
\label{sec:twod}

In this section we study the two dimensional setting. It is clear from the proofs in section \ref{sec:otest} and section \ref{sec:error} that, to get qualitative homogenization, we only needs to prove similar results to Lemma \ref{lem:cell_1}, and to get error and corrector estimates, we only need a quantification like that of Lemma \ref{lem:rvtc}. Those estimates are obtained in Lemma \ref{lem:celld2} below. With those results, the homogenization and corrector results can be proved by the oscillating test function method and its quantification in the earlier sections; we will not repeat those arguments here. 

The technical difference for $d=2$ comes from the energy estimates. In view of the Poincar\'e inequality \eqref{eq:poincare}, the solution to the cell problem \eqref{eq:cell_2} does not satisfy \eqref{eq:cell_C}; instead, we only have
\begin{equation}
\label{eq:en_d2}
\|\nabla \chi^\eta\|_{L^2(\eta^{-1}Q)} \le C|\log \eta|^{\frac12}, \qquad c_\eta = \fint_{\eta^{-1}Q} \chi^\eta \le C|\log \eta|.
\end{equation}
Since $|\log \eta|$ blows up as $\eta \to 0$, the analysis of the last two sections cannot be repeated. 




Inspired by Allaire \cite{Allaire-3}, we set $a_\eta = \frac1{2\pi}(1-\eta^2|T|)$, for each $\eta < 1/4$, and define the function
\begin{equation}
\label{eq:Phi2}
\Phi^\eta(x) = \begin{cases}
0, \quad &x \in B_{1},\\
a_\eta \log |x|, \quad& x \in B_{1/2\eta}\setminus \overline{B}_1,\\
a_\eta |\log \eta/2|, \quad& x \in \frac{1}{\eta}Q \setminus \overline{B}_{1/2\eta}.
\end{cases}
\end{equation}
The reason for this specific definition of $a_\eta$ will be clear later. Since $\overline T \subset B_{1/2}$, we have $\Phi^\eta = 0$ on $T$. Extend $\Phi^\eta$ to be $\eta^{-1}\Z$-periodic. The difference function $w^\eta := \chi^\eta - \Phi^\eta$ then satisfies
\begin{equation}
\label{eq:cell_w2}
\left\{
\begin{aligned}
&-\Delta w^\eta = \eta^2 - 2\eta a_\eta \delta_{S_{1/2\eta}} + a_\eta \delta_{S_1},  &\quad &x \in \eta^{-1} Q\setminus \ol T,\\
&w^\eta = 0, &\quad &x \in \partial T.&\\
\end{aligned}
\right.
\end{equation}
Here, $\delta_{S_r}$, $r = 1/2\eta$ or $1$, is the uniform measure concentrated on the sphere $S_r = \partial B_r(0)$. In other words, we have
\begin{equation*}
\delta_{S_r}(\varphi) := \int_{S_r} \varphi(y) \, d\sigma_y, \qquad \forall \varphi \in C^\infty_c(\R^d).
\end{equation*}
Those measures belong to $H^{-1}(\eta^{-1}Q)$, i.e. the dual space of $H^1(\eta^{-1}Q)$. Set $\nu^\eta := \eta^2 - 2 \eta a_\eta\delta_{S_{1/2\eta}}$. The specific choice of $a_\eta$ yields the following identity:
\begin{equation}
\label{eq:nu_a0}
\langle \nu^\eta, 1 \rangle_{H^{-1}(\eta^{-1}Q\setminus \ol T),H^1(\eta^{-1}Q\setminus \ol T)} = \int_{\eta^{-1}Q\setminus \ol T} \eta^2 - 2\eta a_\eta\delta_{S_{1/2\eta}} \, dy = 0.
\end{equation}

We check that $\|\nabla w^\eta\|_{L^2(\eta^{-1}Q)}$ can be controlled uniformly in $\eta$. 
\begin{lemma}
\label{lem:celld2} Assume {\upshape(P1)} and $d=2$. For each $\eta \in (0,1/4]$, let $\chi^\eta$ be as in \eqref{eq:cell_2} and $\Phi^\eta$ as in \eqref{eq:Phi2}, and set $w^\eta = \chi^\eta - \Phi^\eta$. Then there exists $C > 0$ that is independent of $\eta$ such that
\begin{equation}
\|\nabla w^\eta \|_{L^2(\eta^{-1}Q)} \le C.
\end{equation}
\end{lemma}

\begin{proof} Integrating on both sides of \eqref{eq:cell_w2} against $w^\eta$, we get
\begin{equation}
\label{eq:w1}
\|\nabla w^\eta\|^2_{L^2(\eta^{-1}Q)} = a_\eta\delta_{S_1}(w^\eta) + \nu^\eta(w^\eta).
\end{equation}
For the first term above, we find
\begin{equation}
\label{eq:w2}
\left|\delta_{S_1}(w^\eta)\right| \le C\|w^\eta\|_{L^2(S_1)} \le C\|\nabla w^\eta\|_{L^2(B_2)},
\end{equation}
where we used the usual trace inequality for $H^1$ functions, and the Poincar\'e inequality \eqref{eq:poincare}, and noted that the constants involved are uniform in $\eta$. For the second term, recall that $Y_{f,\eta}$ refers to the set $Q\setminus \eta\ol T$. Then in view of \eqref{eq:nu_a0}, we have
\begin{equation*}
\begin{aligned}
\langle \nu^\eta, w^\eta \rangle_{H^{-1}(\eta^{-1}Y_{f,\eta}),H^1(\eta^{-1}(Y_{f,\eta}))} &= \left\langle \nu^\eta, w^\eta - \fint_{\eta^{-1}Y_{f,\eta}} w^\eta \right\rangle_{H^{-1}(\eta^{-1}Y_{f,\eta}), H^1(\eta^{-1}Y_{f,\eta})}\\
&= \int_{Y_{f,\eta}} \left(1-2a_\eta \delta_{S_{1/2}}\right)[w_\eta - \fint_{Y_{f,\eta}} w_\eta].
\end{aligned}
\end{equation*}
It is clear that $1-2 a_\eta \delta_{S_{1/2}}$ an element in $H^{-1}(Y_{f,\eta})$ and its norm can be bounded uniformly in $\eta$. It follows from the Poincar\'e-Wirtinger inequality that
\begin{equation}
\label{eq:w3}
\left|\langle \nu^\eta, w^\eta \rangle_{H^{-1},H^1} \right| \le C\left\|  w_\eta - \fint_{Y_{f,\eta}} w_\eta \right\|_{H^1(Y_{f,\eta})} \le C\|\nabla w_\eta\|_{L^2(Y_{f,\eta})} = C\|\nabla w^\eta\|_{L^2(\frac{1}{\eta}Y_{f,\eta})}.
\end{equation}
The last step is due to scaling invariance of the $\|\nabla\varphi\|_{L^2}$ norm for $d = 2$.

Combine \eqref{eq:w1}, \eqref{eq:w2} and \eqref{eq:w3}, and recall that $w^\eta$ is extended to $T$ by zero. We finally obtain
\begin{equation}
\label{eq:w4}
\|\nabla w^\eta\|_{L^2(\eta^{-1}Q)} \le C,
\end{equation}
where $C$ is independent of $\eta$. The proof is hence complete.
\end{proof}



\medskip

To build oscillating test functions, we consider a rescaled version of $\chi^\eta$ and define
\begin{equation*}
v^\eps(x) :=  \frac{1}{|\log \eta|} \chi^\eta\left(\frac{x}{\eps \eta}\right), \qquad x \in \eps(Q \setminus \eta \overline{T}).
\end{equation*}
Again, we extend $v^\eps$ to zero in $\eta \ol T$ and extend it further periodically in each cubes of $\eps \Z^d$. Then $v^\eps$ solves the rescaled cell problem:
\begin{equation}
\label{eq:ce2}
\left\{
\begin{aligned}
&-\Delta v^\eps(x) = \frac{1}{|\log \eta|\eps^2}  = \frac{1}{\sigma_\eps^2}, &\quad &x \in \eps \R^2_f,\\
&v^\eps(x) = 0, &\quad &x \in \partial (\eps \R^2_f).
\end{aligned}
\right.
\end{equation}
We have the following properties for $v^\eps$.

\begin{lemma}
\label{lem:celld2} Assume {\upshape ({P1})}, $d= 2$ and $a_\eps \ll \eps$, i.e. $\eta(\eps) = a_\eps/\eps \ll 1$. Let $v^\eps$ be as in \eqref{eq:ce2}. Then for each set $K$ compactly supported in $\R^2$, there exists $C > 0$ that depends only on $d$, $T$ and $K$ such that
\begin{equation}
\label{eq:celld21}
\|\nabla v^\eps\|_{L^2(K)} \le C(\eps|\log \eta|^{\frac12})^{-1}= C\sigma_\eps^{-1}.
\end{equation}
Moreover, we have
\begin{equation}
\label{eq:celld22}
\left\|v^\eps - 1/2\pi\right\|_{L^2(K)} \le C|\log \eta(\eps)|^{-\frac12}.
\end{equation}
Finally, if $\alpha \sim \alpha^\eps_*$, i.e. $\lim_{\eps \to 0} \sigma_\eps \in (0,\infty)$, then
\begin{equation}
\label{eq:celld23}
\nabla v^\eps \rightharpoonup 0 \qquad \text{ weakly in } L^2(K).
\end{equation}

\end{lemma}

\begin{proof} We first observe that \eqref{eq:celld21} and \eqref{eq:celld23} can be proved as in the $d \ge 3$ setting, using the two dimensional versions of the Poincar'e inequality in \eqref{eq:poincare} and the energy estimates \eqref{eq:en_d2}. Consider the proof of \eqref{eq:celld23} for instance, the control $r_\eps$ term in the proof of Lemma \ref{lem:cell_1} becomes, in a typical cell $\eps\square$,
\begin{equation*}
\begin{aligned}
|r_\eps| &\le |\log \eta|^{-1}(\eps \eta)^2 \|D\psi\|_{L^\infty} \|\nabla \chi^\eta\|_{L^2(\eta^{-1}\square)} |\eta^{-1}\square|^\frac12\\
&\le C\eps^2\left(\eta|\log \eta|^{-\frac12}\right).
\end{aligned}
\end{equation*}
This is sufficient for the proof.

The proof of \eqref{eq:celld22} need major modifications in $d=2$. Since $a_\eta$ converges to $1/2\pi$ with error of order $\eta^2$, we only need to estimate $\|v^\eps-a_\eta\|_{L^2(K)}$. As in the proof of Lemma \ref{eq:cell_01}, we use periodicity and reduce the estimates essentially to a rescaled reference cell $\eps Q$. 

In the $\eps$-cell $\eps Q$, extend $v^\eps$ by zero inside $\eps \eta T$. Recall the definitions of $\Phi^\eta$ and $w^\eta$ in \eqref{eq:Phi2}. Then direct computation shows that
\begin{equation}
\label{eq:celld2_0}
v^\eps(x) - a_\eta = \zeta_{\eta,\eps} + \frac{1}{|\log \eta|} \left[ \left(w^\eta\left(\frac{x}{\eps \eta}\right) - \ol w_{\eta,\eps} \right) + \ol w_{\eta,\eps} \right].
\end{equation}
Here, $\zeta_{\eta,\eps}$ is an $\eps Q$-periodic function defined by 
\begin{equation*}
\zeta_{\eta,\eps}(x) := \frac{1}{|\log \eta|} \Phi^\eta\left(\frac{x}{\eps \eta}\right) - a_\eta = \begin{cases}
-a_\eta &\qquad x \in \ol{B}_{\eps\eta},\\
a_\eta|\log \eta|^{-1} \log\left|\frac{x}{\eps}\right| &\qquad x \in \eps\eta(B_{1/2\eta} \setminus \ol B_1),\\
a_\eta |\log 2||\log \eta|^{-1} &\qquad x \in \eps(\square\setminus B_{1/2}),
\end{cases}
\end{equation*}
and $\ol w_{\eta,\eps}$ is a short-hand notation for the average $\fint_{\eps Q} w^\eta(\cdot/\eps\eta)$ in a typical $\eps$-cell. Moreover, by simple scaling check it is clear that $\ol w_{\eta,\eps}$ is the same as the average $\int_{\frac1{\eta}Q} w^\eta$ in the $\frac{1}{\eta}$-cell. Hence, $\ol w_{\eta,\eps}$ satisfies the estimate
\begin{equation*}
|\ol w_{\eta,\eps}|  = \left| \fint_{\eta^{-1}Q} w^\eta(y)\, dy \right| 
\le \frac{\|w^\eta\|_{L^2(\eta^{-1}Q)}}{|\eta^{-1}Q|^{\frac12}} \le \frac{(1/\eta)|\log \eta|^{\frac12}\|\nabla w^\eta\|_{L^2(\eta^{-1}Q)}}{|\eta^{-1}Q|^{\frac12}}.
\end{equation*}
In the last step, we used the Poincar\'e inequality \eqref{eq:poincare}. It follows that, in a typical $\eps$-cell, we have
\begin{equation}
\label{eq:celld2_1}
\|\overline{w}_{\eta,\eps}\|^2_{L^2(\eps Q)} = |\eps Q| |\overline{w}_{\eta,\eps}|^2 \le \eps^2 |\log \eta|\|\nabla w^\eta\|^2_{L^2(\eta^{-1} Q)} \le C\eps^2 |\log \eta|.
\end{equation}
To get the last inequality, we used Lemma \ref{lem:celld2}. This controls the last item on the right hand side of \eqref{eq:celld2_0}. Hence, $\||\log\eta|^{-1}\overline{w}_{\eta,\eps}\|^2_{L^2(\eps Q)}$ is of order $C\eps^2|\log \eta|^{-1}$.

\smallskip

To control the oscillations of $w^\eta(\cdot/\eps\eta)$, we apply the Poincar\'e-Whirtinger inequality and obtain
\begin{equation}
\label{eq:celld2_2}
\begin{aligned}
\|w^\eta(\cdot/\eps\eta) - \overline{w}_{\eta,\eps}\|^2_{L^2(\eps Q)} &= (\eps\eta)^d\|w^\eta(\cdot)- \textstyle\fint_{\eta^{-1} Q} w^\eta\|^2_{L^2(\eta^{-1} Q)}\\
&\le(\eps\eta)^d \eta^{-2}\|\nabla w^\eta\|^2_{L^2(\eta^{-1} Q)} \le C\eps^2.
\end{aligned}
\end{equation}
Hence, $\||\log\eta|^{-1}(w^\eta(\cdot/\eps\eta) - \overline{w}_{\eta,\eps})\|^2_{L^2(\eps Q)}$ is of order $\eps^2|\log \eta|^{-2}$.

\smallskip

Finally, for the $\zeta_{\eta,\eps}$ term, we write $\|\zeta_{\eta,\eps}\|^2_{L^2(\eps Q)}$ as the following sum:
\begin{equation*}
 \|a_\eta\|^2_{L^2(B_{\eps\eta})} + \|a_\eta |\log \eta|^{-1}\log(|\cdot/\eps|)\|^2_{L^2(B_{\eps/2}\setminus \ol B_{\eps\eta})} +  \|a_\eta|\log 2| |\log \eta|^{-1}\|^2_{L^2(\eps Q \setminus \ol B_{\eps/2})}.
\end{equation*}
Then the first term is of order $\eta^2 \eps^2$, and the third term is of order $\eps^2|\log \eta|^{-2}$. For the second term of the above summation, we compute directly and get
\begin{equation*}
\|\log(|\cdot/\eps|)\|^2_{L^2(\eps Q)} \le \eps^2 \|\log (|x|+1)\|^2_{L^2(Q)} \le C\eps^2.
\end{equation*}
Hence, $\|a_\eta |\log \eta|^{-1}\log(|\cdot/\eps|)\|^2_{L^2(B_{\eps/2}\setminus B_{\eps\eta})}$ is of order $\eps^2 |\log \eta|^{-2}$.

The above estimates hold uniformly for all $\eps$-cubes that have non-empty intersections with $K$, and the total number of such cubes is of order $\eps^{-2}$. We hence obtain
\begin{equation}
\|v^\eps - a_\eta\|^2_{L^2(K)} \le C|\log \eta|^{-1}.
\end{equation}
The constant $C$ depends only on $T$ and $K$. This completes the proof of \eqref{eq:celld22}.
\end{proof}



\section*{Acknowledgments}

The author is grateful to Yi-Hsuan Lin, Yong Lu and Christophe Prange for helpful discussions. The work of the author has been supported by the Recruitment Program of Global Experts of China and by the National Natural Science Foundation of China under Grant No.\,11701314.

\appendix

\section{Some useful lemmas}

\medskip


\subsection{A Poincar\'e inequality}

\medskip

We first record a Poincar\'e inequality for $H^1$ functions that vanishes in a set. This inequality is used heavily in our analysis.

\begin{theorem} Assume that $0<a < R$ and $\ol B_a \subset Q_R$. Then there exists a positive constant $C$ depending only on $d$, such that for any $u \in H^1(Q_R)$ satisfying $u = 0$ in $B_a$, we have
\begin{equation}
\label{eq:poincare}
\|u\|_{L^2(Q_R)} \le \begin{cases}
CR\left(\frac{R}{a}\right)^{\frac{d-2}{2}} \|\nabla u\|_{L^2(Q_R)}, \qquad &d\ge 3,\\
CR(\log (R/a))^{\frac12}\|\nabla u\|_{L^2(Q_R)}, \qquad &d=2.
\end{cases}
\end{equation}
\end{theorem}

Clearly, the theorem still hold if the cube is changed to a ball and/or the ball is changed to a cube. This theorem is standard and a proof can be found in \cite{Allaire91-2}. We record the proof here for the convenience of the reader.

\begin{proof} We prove the theorem with $Q_R$ is replaced by $B_R$. By density of smooth functions in $H^1$, we may assume that $u$ is $C^1$. For any point $x \in Q_R\setminus \ol B_a$, we have
\begin{equation*}
	u(x) = u(re) = \int_a^r e\cdot \nabla u(se) ds.
\end{equation*}
Using H\"older inequality, we then have the estimate
\begin{equation*}
	|u(x)|^2 \le \left(\int_a^r |\nabla u(se)|^2 s^{d-1}ds\right) \int_a^r \frac{1}{s^{d-1} }ds.
\end{equation*}
For $d = 2$, the explicit computation
\begin{equation}
\label{eq:p_ec}
	\int_a^r \frac{1}{s^{d-1} }ds = \log r - \log a
\end{equation}
facilitates the estimates below:
\begin{equation*}
\begin{aligned}
\|u\|^2_{L^2(B_R\setminus \ol B_a)} \le& \int_{S^{d-1}} d\sigma_e \int_a^R r^{d-1}\left[\left(\int_a^r |\nabla u(se)|^2 s^{d-1}ds\right) \log\left(\frac{r}{a}\right)\right] dr\\
=& \int_{S^{d-1}} d\sigma_e \int_a^R \left(\int_s^R r\log\left(\frac{r}{a}\right) \,dr\right) \,|\nabla u(se)|^2 s^{d-1}ds\\
\le& \int_{S^{d-1}} d\sigma_e \int_a^R \left(\frac{R^2}{2}\log \frac{R}{a} - \frac{s^2}{2}\log \frac{s}{a}\right) \, |\nabla u(se)|^2 s^{d-1}\, ds\\
\le& \left(\frac{R^2}{2}\log \frac{R}{a} \right)\int_{S^{d-1}} d\sigma_e \int_a^R  \, |\nabla u(se)|^2 s^{d-1}\,ds\\
=&\left(\frac{R^2}{2}\log \frac{R}{a} \right)\|\nabla u\|^2_{L^2(B_R\setminus \ol B_a)}.
\end{aligned}
\end{equation*}
This completes the proof for $d=2$. 

For $d \ge 3$, the explicit computation \eqref{eq:p_ec} is replaced by
\begin{equation*}
	\int_a^r \frac{1}{s^{d-1} }ds = \frac{1}{d-2}\frac{1}{r^{d-2}}\left[ \left(\frac{r}{a}\right)^{d-2}-1 \right] \le \frac{1}{d-2}\frac{1}{r^{d-2}}\left(\frac{r}{a}\right)^{d-2}.
\end{equation*}
We then have
\begin{equation*}
\begin{aligned}
\|u\|^2_{L^2} \le& 
\int_{S^{d-1}} d\sigma_e \int_a^R \left(\int_s^R \frac{r}{d-2}\left(\frac{r}{a}\right)^{d-2} \,dr \right) |\nabla u(se)|^2 s^{d-1} ds\\
\le& \int_{S^{d-1}} d\sigma_e \int_a^R \frac{a^2}{d(d-2)}\left(\frac{R}{a}\right)^d \,|\nabla u(se)|^2 s^{d-1} ds\\
=&\frac{1}{d(d-2)}R^2\left(\frac{R}{a}\right)^{d-2}\|\nabla u\|^2_{L^2}.
\end{aligned}
\end{equation*}
This completes the proof for $d\ge 3$.
\end{proof}

\begin{remark} An easy consequence of the above theorem is, if $v \in H^1_0(D^\eps)$, and $D^\eps$ is the perforated domain satisfying assumption {\upshape(P1)}, then one has
\begin{equation*}
\|\tilde u^\eps\|_{L^2(D)} \le C\sigma_\eps \|\nabla \tilde u^\eps\|_{L^2(D)}.
\end{equation*}
This inequality is most frequently used in this paper and it is verified in the proof of Lemma \ref{lem:u_ee}.
\end{remark}

\medskip

\subsection{Periodic potential theory}
\label{app:lp}

In section \ref{sec:error}, we used \emph{periodic layer potentials} to solve the cell problem \eqref{eq:cell_2} that is posed on the rescaled cube $\frac1\eta Q$, and quantified the convergence of its mean value as $\eta \to 0$. Those potentials are natural generalizations of the classical layer potential operators associated to the Laplace operator in the whole space.

The starting point is to specify the fundamental solution to the Laplace operator in the unit periodic cell, or equivalently in the flat torus $Q_\fp = \R^d/\Z^d$. We seek a solution to
\begin{equation*}
\Delta G(x,y) = \delta_y(x) - 1 \qquad \text{in } Q_\fp,
\end{equation*}
in the distributional sense. Note that the volume of $Q_\fp$ is subtracted on the right hand side; this is necessary for solving the problem, since $Q_\fp$ is a compact manifold without boundary.
Clearly, the solution is unique only up to an additive constant. We hence impose the condition
\begin{equation}
\int_Q G(x,y)dx = 0.
\end{equation}
The existence and uniqueness of $G$ is then classical. As in Remark \ref{rem:periodic}, we view $G$ both as defined on the torus, or periodically over the whole space, in fact the product of them with the diagonal removed. It is easy to check that $G(x;y) = G_0(x-y)$ where $G_0 = G(\cdot;0)$. For notational simplicity, we still use $G$ for $G_0$. It well know that $G(x)$ is a smooth function in $x$ except at the origin, and $G$ has the following form:
\begin{equation*}
 G(x) = \Gamma(x) + R(x), \qquad x\in Q\setminus \{0\}.  
\end{equation*}  
Here, $\Gamma: \R^d\setminus \{0\} \to \R$ is the fundamental solution of the Laplace operator in the whole space given by the formulae:
\begin{equation}
\Gamma(x) = \begin{cases}
\frac{1}{(n-2)c_d} \frac{1}{|x|^{d-2}}, \qquad d \ge 3,\\
-\frac{1}{2\pi}\log |x|, \qquad d = 2,
\end{cases}
\end{equation}
where $c_d$ is the volume of the unit sphere in $\R^d$. Clearly, $R$ is the solution to the equation
\begin{equation*}
\Delta R(x) = -1 \quad\text{in } Q,
\end{equation*}
with proper boundary conditions so that $R + \Gamma$ is periodic. Standard PDE theory then shows that $R$ is smooth in $\ol Q$. Note $R$ itself does not satisfy periodic conditions at $\partial Q$.

To form periodic potentials in the $\eta^{-1}$-periodic cell $\eta^{-1}Q_{\fp}$, we rescale $G$ and define $G^\eta(x;y) = \eta^{d-2}G(\eta(x-y))$, for $x,y \in \frac{1}{\eta}Q$ and $x\ne y$. It is the unique periodic fundamental solution to
\begin{equation}
\label{eq:Geta}
\Delta G^\eta(x;y) = \delta_y(x) - \eta^d \quad\text{in } \frac1\eta Q_\fp \qquad\text{and} \qquad \int_{\frac1\eta Q} G^\eta(x;y) = 0.
\end{equation}
Then $G^\eta$ has the decomposition formula:
\begin{equation}
\label{eq:Geta_d}
G^\eta(x) = \Gamma(x) + \eta^{d-2}R(\eta x), \qquad x\in \eta^{-1}Q\setminus \{0\}.
\end{equation}

\medskip

For each fixed $\eta > 0$, we define the periodic single layer potential operator on $\frac{1}{\eta} Q_\fp$ by
\begin{equation}
\label{eq:psingle}
\cS^{(\eta)}_{T,\fp} [\varphi](x) := \int_{\partial T} G^\eta(x-y)\varphi(y) d\sigma(y), \qquad x\in \frac{1}{\eta}Q_\fp \setminus \partial T,
\end{equation}
and define the periodic double layer potential by
\begin{equation}
\label{eq:pdouble}
\cD^{(\eta)}_{T,\fp} [\varphi](x) := \int_{\partial T} \nu_y \cdot \nabla G^\eta(x-y) \varphi(y) d\sigma(y), \qquad x \in \frac{1}{\eta} Q_\fp \setminus \partial T,
\end{equation}
where $\varphi \in L^2(\partial T)$. In this paper, we only treat the case where the set $T$ has regular boundary, i.e. under the following assumption.
\begin{itemize}
\item[(T1)] $T$ is open bounded and \emph{connected} and $\partial T$ is $C^{1,\alpha}$, for some $\alpha \in (0,1)$.
\end{itemize}

In view of \eqref{eq:Geta_d}, the integrals above are regular integrals. It is easy to check that, $\cD^{(\eta)}_{T,\fp}[\varphi]$ is harmonic in $T$ and in $\eta^{-1}Q_\fp\setminus \ol T$ for any $\varphi \in L^2(\partial T)$, and the same holds for $\cS^{(\eta)}_{T,\fp}[\psi]$ but only if the further condition $\int_{\partial T} \psi = 0$ is imposed. We denote by $L^2_0(\partial T)$ the mean zero subspace of $L^2(\partial T)$.

In view of the decomposition formula \eqref{eq:Geta_d}, we find that, in some sense, the periodic potentials are ``perturbations'' to the classical potentials associated to the Laplace operator in the whole space. More precisely,
\begin{equation*}
\cS^{(\eta)}_{T,\fp}[\varphi] = \cS_T[\varphi] + \eta^{d-2} \cR_1[\varphi], \qquad \cD^{(\eta)}_{T,\fp}[\varphi] = \cD_T[\varphi] + \eta^{d-1} \cR_2[\varphi], \quad \text{in } \quad \frac{1}{\eta} Q_\fp\setminus \partial T.
\end{equation*}
where $\cS_T$ and $\cD_T$ are the classical single and double layer potentials respectively, defined similarly to \eqref{eq:psingle} and \eqref{eq:pdouble} with $G^\eta$ replaced by $\Gamma$. The operators $\cR_1$ and $\cR_2$ are well defined by
\begin{equation}
\label{eq:Rijdef}
\cR_1[\varphi](x) = \int_{\partial T} R(\eta(x-y))\varphi(y)d\sigma_y \quad \text{and} \quad \cR_2[\varphi](x) = \int_{\partial T} \nu_y \cdot \nabla R(\eta(x-y)) \varphi(y) d\sigma_y.
\end{equation}
In view of the smoothness of $R$, the operators $\cR_1$ and $\cR_2$ are smoothing.

We refer to \cite{Folland,AmmKan} for some practical guide to layer potentials techniques in solving Laplace equations, and to \cite{FabJod,Verchota} for some original developments. Many results there for the classical layer potentials can be carried out almost in parallel for periodic potentials. Here are some basic results.

\begin{theorem}\label{thm:dlp} Assume {\upshape(T1)}. Then for any $\psi \in L^2(\partial T)$, the following results hold.
\begin{enumerate}
\item The traces of the double layer potential $\cD^{(\eta)}_{T,\fp}[\varphi]$ from the outer and inner side of $\partial T$ exist and satisfy the jump relation
\begin{equation}
\label{eq:PNop}
\left.\cD^{(\eta)}_{T,\fp}[\psi] \right\rvert_{\pm}(x) = \lim_{t \to 0+} \cD^{(\eta)}_{T,\fp}[\psi](x+t\nu_x) = \mp \frac12 \psi(x) + \cK^{(\eta)}_{T,\fp}[\psi](x), \qquad x \in \partial T
\end{equation} 
where $\cK^{(\eta)}_{T,\fp}$ is the generalized Neumann-Poincar\'e operator defined by
\begin{equation}
\cK^{(\eta)}_{T,\fp}[\psi](x) = \int_{\partial T} \nu_x \cdot \nabla_y G^{\eta}(x-y) \psi(y) dy.
\end{equation}
For the single layer potential, we have
\begin{equation}
\cS^{(\eta)}_{T,\fp}[\psi] \big\lvert_+(x) =  \cS^{(\eta)}_{T,\fp}[\psi] \big\lvert_-(x), \qquad x \in \partial T.
\end{equation}
\item The inner and outer normal derivatives, that is
\begin{equation*}
\left.\frac{\partial}{\partial \nu}\right\rvert_\pm \cD^{(\eta)}_{T,\fp}[\psi](x) = \lim_{t\to 0+} \nu_x \cdot \nabla \cD^{(\eta)}_{T,\fp}[\psi](x+t\nu_x), \qquad x \in \partial T
\end{equation*}
exist and agree on $\partial T$. On the other hand, for the single layer potential, we have
\begin{equation}
\left.\frac{\partial}{\partial \nu}\right\rvert_\pm \cS^{(\eta)}_{T,\fp}[\psi](x) = \pm \frac{1}{2} \psi(x) + (\cK^{(\eta)}_{T,\fp})^*[\varphi](x), \qquad x \in \partial T.
\end{equation}
\end{enumerate}
\end{theorem} 

Here and in the sequel, the adjoint of an operator $\mathcal{A}$ is denoted by $\mathcal A^*$. For simplicity, we use the short-hand notation $\cK^\eta_T$ for $\cK^{(\eta)}_{T,\fp}$; the short-hand notations $\cS^\eta_T$ and $\cD^\eta_T$ are understood similarly. The theorem above can be proved easily following the usual arguments; see in particular \cite[Section 2.8]{AmmKan}. In view of \eqref{eq:Geta_d} again, $\cK^\eta_T$ can be viewed as a perturbation to the classical Neumann-Poincar\'e operator $\cK_T$, which is defined as in \eqref{eq:PNop} with $G^\eta$ replaced by $\Gamma$. More precisely, we have
\begin{equation}
\label{eq:NPop-d}
 \cK^\eta_T [\varphi] = \cK_T[\varphi] + \eta^{d-1}\cR_2[\varphi], \qquad \text{on } \partial T.
\end{equation}
Here, $\cR_2$ is defined as before but with $x \in \partial T$; it well defined in view of the smoothness of $\partial T$. It is a different operator but we abuse notations and still denote it by $\cR_2$. 

Thanks to the regularity of $\partial T$, the integrals in $\cK_T$ and $\cK^\eta_T$ are regular. Those operators map $L^2(\partial T)$ to $H^1(\partial T)$ and hence are compact in $L^2(\partial T)$. It follows that $\pm \frac12 I + \cK_T$ is Fredholm; the same statement holds for $-\frac12 I + \cK^\eta_T$ and for their adjoint operators. Those results can be found in e.g. \cite[Chapters 1 and 3]{Folland}.

\medskip


Finally, we present some further mapping properties regarding the Neumann-Poincar\'e operators.

\begin{theorem}\label{thm:layerp} Assume {\upshape(T1)} and $d\ge 3$. Then the following results hold.
\begin{enumerate}
\item $\mathrm{Ker}(-\frac12 I + \cK_T)$ is one dimensional and spanned by the constant function $\psi_* \equiv 1$.
\item $\mathrm{Ker}(-\frac12 I + \cK_T^*)$ is one dimensional and spanned by the unique function $\varphi_* \in L^2(\partial T)$ which satisfies: $(\varphi_*,a_*)$ is the unique pair of elements in $L^2(\partial T)\times \R$ verifying
\begin{equation*}
\mathcal S_T [\varphi_*] + a_* = 0 \quad \text{in } \R^d\setminus \overline{T}\qquad\text{and}\qquad \int_{\partial T} \varphi_* = 1.
\end{equation*}
Here, $\cS_T$ is understood as the trace of the single layer potentials from the exterior of $T$. 
\item The decomposition $L^2(\partial T) = \mathrm{Ran}(-\frac12 I + \cK_T) \oplus \mathrm{Ker}(-\frac12 I + \cK_T)$ holds.
\item The operator $-\frac12 I + \cK_T : L^2_0(\partial T) \to \mathrm{Ran}(-\frac12 I + \cK_T)$ is invertible.
\end{enumerate}
\end{theorem}

The results above are classical: item one is proved in \cite[Chapter 3]{Folland}; the characterization in item two can be found in \cite[Theorem 2.26]{AmmKan}; item three is proved in \cite[Corollary 3.39]{Folland}. Using Fredholm theory, the fourth term can be easily proved by showing the injectivity of the operator. We remark that the dimensions determined above are due to the assumption that $T$ is connected, which is imposed only for simplicity. If $T$ has multiple connected components, the above theorem can be generalized accordingly.

\begin{corollary}\label{cor:Gam}
Assume {\upshape(T1)} and $d\ge 3$. Then, for any $h \in L^2(\partial T)$, it can be written as
\begin{equation*}
h = h_1 + h_0, \qquad h_0 : = \langle h,\varphi_* \rangle_{L^2,L^2} \in \R \quad\text{and}\quad h_1 \in \mathrm{Ran}(-\frac{1}{2}I + \cK_T).
\end{equation*}
In particular, $\langle \Gamma\big\vert_{\partial T}, \varphi_* \rangle = - \frac{1}{\mathrm{Cap}(T)}$.
\end{corollary}
\begin{proof} The decomposition of $h$ follows from item two of the previous theorem. To find the constant component $h_0$, we use the fact that $\mathrm{Ran}(-\frac12 I + \cK_T)$ is closed (by Fredholm theory) and that it is orthogonal to $\mathrm{Ker}(-\frac12 I + \cK_T^*)$. By item three of the previous theorem, we get the result.

To find $\langle\Gamma\big\rvert_{\partial T},\varphi_*\rangle$, we first recall the fact that $a_* = 1/\mathrm{Cap}(T)$; see e.g. \cite{AmmKan}. Also, by the continuity of $\cS_T[\varphi_*]$ across $\partial T$ and that $\cS_T[\varphi_*]$ is harmonic in $T$, we conclude that $\cS_T[\varphi_*] = -a_*$ in $T$. Then by definition and the symmetry of $\Gamma(\cdot,\cdot)$, we have
\begin{equation*}
\langle\Gamma\big\rvert_{\partial T},\varphi_*\rangle = \int_{\partial T} \Gamma(0-y)\varphi_*(y) d\sigma_y = \cS_T[\varphi_*](0) = -a_* = -\frac{1}{\mathrm{Cap}(T)}. 
\end{equation*}
This completes the proof.
\end{proof}

Finally, we prove the following mapping property for the periodic double layer potential.

\begin{theorem}\label{thm:map_dl} Assume {\upshape(T1)} and $d\ge 3$. Then $\mathrm{Ran}(-\frac12 I + \cK^\eta_T) = L^2(\partial T)$. In particular, 
\begin{equation*}
(-\frac12 I + \cK^\eta_T)^{-1}[1] = (-\eta^d|T|)^{-1}.
\end{equation*}
\end{theorem}

\begin{proof} First, by direct computation we check that $\cK^\eta_T[1] = \frac12 -\eta^d|T|$, so we only need to prove the first statement of the theorem. By the compactness of $\cK^\eta_T$ in $L^2(\partial T)$, the range of $-\frac12 I + \cK^\eta_T$ is closed and equals the orthogonal complement of $\mathrm{Ker}(-\frac12 I + (\cK^\eta_T)^*)$. Take an arbitrary element $\varphi$ from this kernel, then we have
\begin{equation*}
0 = \int_{\partial T} (-\frac12 I + (\cK^\eta_T)^*)[\varphi] = \langle (-\frac12 I + \cK^\eta_T)[1], \varphi\rangle_{L^2,L^2} = -\eta^d|T| \int_{\partial T} \varphi.
\end{equation*}
We conclude that $\varphi \in L^2_0(\partial T)$. It follows that $\cS^\eta_T[\varphi]$ is harmonic in $\frac{1}{\eta}Q_\fp \setminus \ol T$ and in $T$. Let $u_i$ and $u_e$ denote the restriction of $\cS^\eta_T[\varphi]$ in and outside $T$. We then have
\begin{equation*}
\int_T |\nabla u_i|^2 = \int_{\partial T} u_i \left.\frac{\partial u_i}{\partial \nu}\right\rvert_- = \int_{\partial T} u_i \left(-\frac12 I + (\cK^\eta_T)^*\right)[\varphi] = 0.
\end{equation*}
Hence, $u_i = C$ is a constant function. Similarly, we have
\begin{equation*}
\begin{aligned}
\int_T |\nabla u_e|^2 &= -\int_{\partial T} u_e \left.\frac{\partial u_i}{\partial \nu}\right\rvert_+ = -\int_{\partial T} u_e \left(\frac12 I + (\cK^\eta_T)^*\right)[\varphi]\\
& = -\int_{\partial T} u_e \varphi = -\int_{\partial T} u_i \varphi = -C\int_{\partial T} \varphi = 0.
\end{aligned}
\end{equation*}
In the second line above, we used the continuity of $\cS^\eta_T [\varphi]$ across $\partial T$, and that $\varphi \in L^2_0$. It follows that $u_e = u_i = C$, and by the jump relation in Theorem \ref{thm:dlp} we get $\varphi = \partial_\nu u_e\big\rvert_+ - \partial_\nu u_i\big\rvert_- = 0$. We conclude that $\mathrm{Ker}(-\frac12 I + (\cK^\eta_T)^*)$ contains only $0$, and that $\mathrm{Ran}(-\frac12 I + \cK^\eta_T)$ is the whole space.
\end{proof}

\bibliographystyle{abbrv}
\bibliography{j_ph}
\end{document}